\documentclass[11pt]{article}
\pdfoutput=1

\usepackage[T1]{fontenc}
\usepackage[utf8]{inputenc}
\usepackage{lmodern}
\usepackage{microtype}

\usepackage{amsmath,amssymb,amsthm,latexsym,mathtools}
\usepackage{float}
\usepackage{booktabs}
\usepackage{multirow}
\usepackage{graphicx}
\usepackage{xcolor}
\usepackage{booktabs}
\usepackage{nicematrix}
\usepackage[ruled,vlined]{algorithm2e}
\usepackage{enumitem}
\usepackage[left=1in,right=1in,top=0.8in,bottom=0.9in]{geometry}
\usepackage{changepage}
\usepackage{hyperref}
\usepackage{cleveref}
\usepackage{float}
\usepackage{changepage} % provides adjustwidth

\usepackage[numbers,sort&compress]{natbib}

\newcommand{\hl}[1]{#1}

\newcommand{\Prob}{\ensuremath{\mathbb{P}}}
\newcommand{\Ex}{\ensuremath{\mathbb{E}}}
\newcommand{\transp}[1]{{#1}^{\scriptscriptstyle{\mathsf{T}}}}
\newcommand{\argmax}{\mathop{\rm arg\,max}}
\newcommand{\argmin}{\mathop{\rm arg\,min}}

\DeclareMathOperator\sgn{sgn}

\providecommand{\natexlab}[1]{#1}

\theoremstyle{plain}
\newtheorem{theorem}{Theorem}
\newtheorem{lemma}{Lemma}
\newtheorem{proposition}{Proposition}

\theoremstyle{definition}
\newtheorem{definition}{Definition}
\newtheorem{assumption}{Assumption}

\title{A fast-pivoting algorithm for Whittle's restless bandit index}
\author{Jos\'e Ni\~no-Mora\\
Departamento de Estad\'{\i}stica \\ Universidad Carlos III de Madrid\\
28903 Getafe (Madrid), Spain\\
\texttt{jose.nino@uc3m.es}}
\date{}

\begin{document}

\maketitle

\begin{abstract}
The Whittle index for restless bandits (two-action semi-Markov decision processes) provides an intuitively appealing optimal policy for controlling a single generic project that can be active (engaged) or passive (rested) at each decision epoch, and which can change state while passive. It further provides a practical heuristic priority-index policy for the computationally intractable multi-armed restless bandit problem, which has been widely applied over the last three decades in multifarious settings, yet mostly restricted to project models with a one-dimensional state.
This is due in part to the difficulty of establishing indexability (existence of the index) and of computing the index for projects with large state spaces. 
This paper draws on the author's prior results on sufficient indexability conditions and an adaptive-greedy algorithmic scheme for restless bandits to 
obtain a new fast-pivoting algorithm that computes the $n$ Whittle index values of an $n$-state restless bandit  by performing, after an initialization stage, $n$ steps that entail $(2/3) n^3 + O(n^2)$ arithmetic operations. This algorithm also draws on the parametric simplex method, and is based on elucidating the pattern of parametric simplex
 tableaux, which allows to exploit special structure to substantially simplify and reduce the complexity of simplex pivoting steps.
A numerical study demonstrates
 substantial runtime speed-ups versus alternative algorithms.
\end{abstract}



\medskip
\noindent\textbf{Note:} Published in \emph{Mathematics} \textbf{7}, 2226 (2020). DOI:\ \url{https://doi.org/10.3390/math8122226}

\section{Introduction}
\label{s:intro}
We consider a general two-action{\hl{ (1: engaged/active; 0: rested/passive) semi-Markov decision process (SMDP) restless bandit}} model
%Are the italics necessary in this manuscript? if yes, please explain. if not, please remove.
%Answer: Not strictly necessary in this instance, I removed them
 (see, for example, (\cite{put94}, Ch.\ 11) and~\cite{whit88})
of a dynamic and stochastic \emph{project},    
 whose
 state
 $X(t)$ moves over continuous time $t \in [0, \infty)$
across the \emph{state space} $\mathcal{N}$, which is assumed finite consisting of $n \triangleq |\mathcal{N}|$ states.
At each of
an increasing sequence of \emph{decision epochs} 
$\{t_k\}_{k=0}^\infty$ starting at $t_0 = 0$ and with $t_k \nearrow \infty$ as $k \nearrow \infty$,
the \emph{embedded state} $X_k \triangleq
X(t_k)$ is observed and then an \emph{action}
$A_k \triangleq A(t_k) \in \{0, 1\}$ is chosen, which remains fixed over the subsequent \emph{stage} $[t_k, t_{k+1})$, so $A(t) = A(t_k)$ for $t \in [t_k, t_{k+1})$.
Note that the project state $X(t)$ can change within such stages, and~that processes $X(t)$ and 
$A(t)$ are piecewise constant and~right-continuous.

When the project is in state $X(t) = i$ under action $A(t)  = a$, it accrues
rewards and consumes a generic resource at rates $R_i^a$ and $Q_i^a$ per unit time, respectively, which are exponentially discounted with rate $\alpha >0$. Resource consumption rates are assumed to satisfy  the natural requirement that $0 < Q_i^1 \geqslant Q_i^0 \geqslant 0$, so when the project is active, it consumes resources at a positive rate, which is not lower than when it is~passive.

To select actions, a
\emph{policy} $\pi$ is adopted from the class $\Pi$ of 
history-dependent randomized policies, 
which base the choice of action at each decision epoch $t_k$ on the history 
$\mathcal{H}_k \triangleq \{(X_{k'}, A_{k'})\colon 0 \leqslant k' < k\} \cup \{X_k\}$
of embedded states and~actions.

Suppose that the amount of resource consumed by the project must be paid for at the unit price $\lambda$, so, writing as  $\Ex_i^\pi[\cdot]$ the
expectation starting from $X(0) = i$ under $\pi$,  
\[
V_{i}^\pi(\lambda) \triangleq \Ex_i^\pi\Big[\int_{0}^\infty \big(R_{X(t)}^{A(t)} - \lambda Q_{X(t)}^{A(t)}\big) e^{-\alpha t} \, dt\Big]
\]
is the \emph{$($expected total discounted$)$ project value} starting from $i$ under $\pi$, and~
$V_{i}^*(\lambda) \triangleq \sup_{\pi \in \Pi} \, V_{i}^\pi(\lambda)$ is the 
\emph{optimal project value} starting from state $i$.
Consider now the \emph{$\lambda$-price problem}
\begin{equation}
\label{eq:nuwpro}
(P_\lambda)\colon \quad \textup{find} \enspace \pi_\lambda^* \in \Pi \enspace \textup{such that} \enspace V_{i}^{\pi_\lambda^*}(\lambda)  = V_{i}^*(\lambda) \enspace \textup{for every state} \enspace i \in \mathcal{N},
\end{equation}
which entails finding a policy that maximizes the project value for every initial state. We will 
call \emph{$P_\lambda$-optimal} a policy $\pi_\lambda^*$ solving the $\lambda$-price problem $(P_\lambda)$ in (\ref{eq:nuwpro}).

Under mild assumptions, it is well-known from the theory of finite-state and -action SMDPs (\mbox{see (\cite{put94}, Ch.\ 11)}) that, to~solve  problem (\ref{eq:nuwpro}), it suffices to consider the class $\Pi^{\textup{SD}}$ of 
\emph{stationary deterministic policies}, which prescribe a fixed action at each decision epoch based on the current project state. For~any given resource price $\lambda$, 
a $P_\lambda$-optimal policy in $\Pi^{\textup{SD}}$ can be computed by
classic algorithms, such as value or policy~iteration. 

Instead, we will treat the resource price $\lambda$ as a parameter, and~consider a solution approach for the 
parametric collection $\mathcal{P}$ of \emph{all} $\lambda$-price problems $(P_\lambda)$ as $\lambda \in \mathbb{R}$.
Thus, let us say that the project, or~more precisely, the parametric problem collection $\mathcal{P}$, is \emph{indexable} if, for every project state $i$, there exists a finite critical price $\lambda_i^*$ such that, for~any problem $(P_\lambda)$ and at a decision epoch in which the project is in state $i$: 
(i) it is optimal to \emph{engage} the project if, and only if, $\lambda_i^* \geqslant \lambda$; and (ii) it is optimal to \emph{rest} the project if, and only if, $\lambda_i^* \leqslant \lambda$. Hence, both actions will be optimal in state $i$ if, and only if,  $\lambda_i^* = \lambda$.

We will refer to the mapping $i \mapsto \lambda_i^*$ as the project's \emph{Whittle index}, since it was Whittle who introduced such a concept in~\cite{whit88}, in~a Markovian setting with resource consumption $Q_i^a \triangleq a$. As~for the extension to general resource consumption $Q_i^a$, it was introduced in~\cite{nmmp02}.
In the case, $Q_i^a \triangleq a$, the~Whittle index extends the \emph{Gittins index} for classic (non-restless) bandit projects, which do not change state while passive. 
Thus, strictly speaking, the~index considered herein for general $Q_i^a$ is a \emph{generalized Whittle index}, which reduces to a 
\emph{generalized Gittins index} in the non-restless~case.

Considering a semi-Markov instead of a purely Markovian setting significantly expands the modeling power of the resulting restless bandits, as~in some applications, the time during which a project remains engaged before the next decision can be made may follow a general distribution. 
As an example (see~\cite{nmmor06}), imagine that a ``project'' represents a queue of jobs with Poisson arrivals and generally distributed service times, where serving a job (active action) is non-preemptive, \mbox{i.e.,~it cannot} be interrupted once started. In~\cite{fuetal16}, semi-Markov restless bandits are used to model 
dynamic job assignment for energy-efficient server farms.
Semi-Markov restless bandits can also be used to model classic (non-restless) bandits where changing from engaging one project to another entails project-dependent switching times, as~shown in~\cite{nmvaluet207}.

Besides its intrinsic interest for solving the aforementioned single-project parametric problem collection $\mathcal{P}$, Whittle proposed in~\cite{whit88} to use the index $\lambda_i^*$ as the basis of a widely popular heuristic for the \emph{multi-armed restless bandit problem} (MARBP), in~which $M$ out of $N > M$ restless bandit projects must be selected to be engaged at each time to maximize the value (under a discounted or long-run average criterion) earned from the $N$ projects over an infinite horizon.
For a sample of recent applications, see, for example,~\cite{qianetal16,fuetal16,borkarPatta17,borkarRavikSab17,borkaretal18,gerumetal19,abbuMakis19,ayeretal19,nmcor19,fuMoran20,hsuetal20,sunetal20,lietal20}.
While the MARBP is computationally intractable (\hl{PSPACE-hard}; %it's bold necessarey? if not, please remove. If yes, please explain.
% Answer: some authors use boldface  for complexity classes, but it's not strictly necessary; I removed it
see~\cite{papTsik99}), the~\emph{Whittle index policy} is an intuitively appealing heuristic where, at each time, $M$ projects with the highest current indices are engaged, so the Whittle index plays the role of a \emph{priority index} for a project to be engaged.
This policy is convenient for practical implementation, as~it avoids the \emph{curse of dimensionality}, since each project has its own Whittle index. Furthermore, an~extensive body of numerical evidence accumulated over the last three decades has shown that the policy is often nearly optimal. 
Though its exact analysis remains elusive, Whittle's conjecture that his proposed policy enjoys a form of asymptotic optimality has been established under certain conditions. See~\cite{wewe90,wewe91,ouyangetal16,verloop16}.

Yet, unlike the Gittins index, which is well-defined for any classic bandit, the~Whittle index exists only for a limited type of restless bandits, which are called \emph{indexable}.
{\hl{Typically, researchers use ad~hoc analyses to  prove indexability  and calculate the Whittle index in particularly models.}}
% To convey the intended meaning I changed ``In their most published work, researchers prove indexability, particularly models, and calculate the Whittle index based on s'' by the text shown
In contrast, the~author has introduced and developed in~\cite{nmaap01,nmmp02} a methodology to establish indexability and compute the Whittle index for general finite-state restless bandits, extended to the semi-Markov denumerable-state case in~\cite{nmmor06} and to the continuous-state case in~\cite{nmmor20}.
\mbox{The effectiveness} of such an approach, based on verification of so-called \emph{PCL-indexability conditions}---as they are grounded on satisfaction by project performance metrics of \emph{partial conservation laws} (PCLs)---has been demonstrated in diverse models. See, for example,~\cite{nmqs06,nmtop07,caonyb08,hubWu08,nmijoc08,nmcor12,nmejor12,heetal12,mennerZeil18,nmcor19,danceSi19}.
 
In the case of finite-state restless bandits,~\cite{nmaap01,nmmp02} introduced the concept of satisfaction of PCLs by project performance metrics, along with a related one-pass \emph{adaptive-greedy index algorithm}, \mbox{which calculates} the $n$ index values of an $n$-state \emph{PCL-indexable}  project in $n$ steps.
This has its early roots in Klimov's algorithm in~\cite{kl74} for calculating the optimal priority indices for scheduling a multiclass queue with Bernoulli feedback, which was extended in~\cite{nm96} to a framework of stochastic scheduling systems satisfying so-called \emph{generalized conservation laws}, such as the classic (non-restless) multi-armed bandit problem and branching bandits. See also~\cite{nmconsLawsEORMS,nmKlimovsEORMS}.
Yet, the aforementioned work has not addressed the efficient computational implementation of such an algorithm, which is necessary for its actual deployment and widespread~application. 
 
This paper develops an efficient computational scheme for calculating the Whittle index of a general finite-state PCL-indexable restless bandit, by~extending the approach in~\cite{nmijoc07} from discrete-time classic bandits to semi-Markov restless bandits. 
The main contribution is a new
fast-pivoting block implementation of the adaptive-greedy Whittle index algorithm in~\cite{nmaap01,nmmp02} that performs---after an initialization stage that entails solving a block $n \times n$ linear equation
system---$(2/3) n^3 + O(n^2)$ arithmetic operations.
The complexity of the resulting algorithm can be further
reduced in particular models, by~exploiting the special structure of
the underlying state space and 
transition matrices. 
However, it appears unlikely that such an operation count can be further reduced for general restless bandits, 
since as shown in~\cite{nmaap01,nmmp02} computing the Whittle index, even if the ordering of the project states was known in advance, entails the solution of an $n \times n$ linear equation system, which would be solved in $(2/3) n^3 + O(n^2)$ operations by Gaussian~elimination.

Note that this is the fastest operation count for a general Whittle index algorithm presented to date. 
In contrast, the~Whittle index algorithm in~\cite{ayestetal20}, which applies only to the average criterion, has a complexity of $O(n^4 2^n)$ operations, which reduces to $O(n^5)$ for {\hl{projects with a one-dimensional state that are known beforehand to be both indexable and solved optimally by 
 threshold policies.}}
 % To convey the intended meaning with greater clarity I changed the text ``indexable bandits, in which it is known that  threshold policies are~optimal.'' to the text shown.

\subsection*{Structure of the~Paper}
\label{s:sofp}
The rest of the paper proceeds as follows.
Section~\ref{s:litrev} presents a review of the related literature.
Section~\ref{s:rbi} reviews previous results on Whittle indexation for
finite-state restless bandits via PCL-indexability. 
Section~\ref{s:ecge} lays the groundwork for an efficient implementation of the adaptive-greedy index algorithm, drawing on \emph{dynamic programming} (DP) and 
\emph{linear programming} (LP) methods.
Section~\ref{s:fpia} applies such results to develop a  fast-pivoting computational implementation of the
adaptive-greedy index algorithm.
Section~\ref{s:rac} outlines how to extend the previous
results to the average criterion.
Section~\ref{s:ce} presents the results of a numerical
study testing the runtimes of several index algorithms and demonstrating that the proposed fast-pivoting algorithm has significantly lower runtimes than alternative algorithms.
Section~\ref{s:disc} discusses the results presented in the paper.
Section~\ref{s:concl} concludes the~paper.

\section{Review of Related~Literature}
\label{s:litrev}
In this section, we review related literature, focusing on relatively recent work.
We refer the reader to the recent monograph~\cite{zhao20} on multi-armed bandits, which discusses both the MABP and the MARBP and their widespread~applications.

When projects are classic (non-restless), the~Whittle index reduces to the Gittins index. 
The efficient computation of the Gittins index has been extensively investigated in the literature.
See, for example,~\cite{vawabu85,chenka86,kallenb86,katehVein87,kattaset04,sonin08}.
While some algorithms proposed in the aforementioned papers perform $O(n^3)$ arithmetic operations for a general $n$-state bandit,~\cite{nmijoc07} presents a \emph{fast-pivoting} implementation of the adaptive-greedy Gittins index algorithm in~\cite{nm96} that performs
 $(2/3)n^3 + O(n^2)$ arithmetic operations, thus achieving lower
complexity than alternative algorithms and matching that of Gaussian elimination for solving an \mbox{$n \times n$ linear} equation system.
It is unlikely that such a complexity count can be improved, since it is shown in~\cite{nm96} that computing the Gittins index reduces precisely to solving an $n \times n$ linear equation system whose solution is subject to certain inequalities. 
The algorithm in~\cite{nmijoc07} is based on an elucidation of the pattern of parametric simplex
 tableaux and exploitation of special structure to lower
 the computational effort of simplex pivoting steps with respect to standard~pivoting.
 
 In contrast, efficient computation of the Whittle index for restless bandits has received relatively scant research attention. 
Three approaches can be distinguished in the literature: 
deriving the Whittle index in closed form, iterative index approximation, and~exact numerical computation.
The first approach is to derive the Whittle index in closed form, as, for example,~ in~\cite{fuetal16,ayeretal19,fuMoran20,hsuetal20,sunetal20,lietal20}. 
This is to be preferred whenever possible, as~the resulting analytical expressions for the Whittle index, besides~facilitating its numerical evaluation, may provide valuable insight on the index dependence on model parameters. 
However, obtaining closed-form Whittle index formulae is only possible in relative simple models, typically with a one-dimensional~state.

In more complex models in which the Whittle index cannot be evaluated in closed form, the~most widespread approach, which has its roots in the calibration method for the Gittins index in~\cite{gi79}, \mbox{is to} apply an iterative procedure for approximately computing the index. This is done, for example, in~\cite{qianetal16,borkarPatta17,borkarRavikSab17,borkaretal18,gerumetal19,abbuMakis19}.
Besides the drawback that the resulting index is only an approximation to the true Whittle index, \mbox{this approach} is typically computationally~expensive.

As for the third approach, efficient exact numerical computation of the Whittle index, this has received the least research attention to date.
In~\cite{nmaap01,nmmp02}, the author introduced an \emph{adaptive-greedy algorithm} for computing the Whittle index in 
general finite-state restless bandits, provided that they satisfy the PCL-indexability conditions also introduced in such work.
The algorithmic computes the $n$ Whittle index values of an $n$-state bandit in $n$ steps, proceeding in a greedy fashion at each step.
However, as~given in~\cite{nmaap01,nmmp02}, such a method provides only an \emph{algorithmic scheme}, as~it is not specified how certain metrics that arise in its description are to be computed in practice.
The effectiveness of such an approach is demonstrated in~\cite{nmmp02} in the setting of a broad birth--death restless bandit model, both under the discounted and long-run average criteria, motivated by queueing admission control and routing problems, where a specific implementation of the adaptive-greedy algorithm for such a case is obtained that computes the first $n$ Whittle indices in $O(n)$ operations, under~mild conditions on model parameters. 
An alternative approach to long-run average birth--death restless bandits is developed in~\cite{larraetal16}. Yet, to~prove indexability, Proposition 2 in that paper assumes both that threshold policies are optimal (which in the birth--death model in~\cite{nmmp02} is obtained as a byproduct of the PCL-indexability conditions) and that a certain function of steady-state probabilities is strictly increasing, which is nontrivial to verify. Further, the Reference~\cite{larraetal16} does not give an index-computing algorithm, but~an expression for the Whittle index in terms of steady-state metrics, which also appears in~\cite{nmmp02}.

A Whittle index computing algorithm for general continuous-time finite-state restless bandits has been recently proposed in~\cite{ayestetal20}, focusing on the average criterion, as~it is stated there that the approach does not apply to the discounted criterion. As~for the arithmetic-operation complexity of the Whittle index algorithm in~\cite{ayestetal20}, which also checks for indexability, Remark 3.5 states that it performs $O(n^4 2^n)$ operations. 
 Even in the case that indexability is known to hold and that threshold policies are optimal, it is stated in Remark 4.1 of that paper that the complexity of that index algorithm reduces to $O(n^5)$ operations.
 This is to be contrasted with the Whittle index algorithm presented herein, which has a cubic operation complexity in the number $n$ of bandit~states.

%%%%%%%%%%%%%%%%%%%%%%%%%%%%%%%%%%%%%%%%%%
\section{Review of Finite-State Restless Bandit Whittle Indexation via~PCL-Indexability}
\label{s:rbi}
This section reviews key results of the author's
approach to restless bandit indexation, as~it applies to a single finite-state
semi-Markov restless bandit project, which we
 will simply refer to as a ``project'' in the~sequel.

\subsection{SMDP Restless Bandits and Their Discrete-Stage~Reformulation}
\label{s:dtr}
Consider  a SMDP restless bandit project, as outlined in Section~\ref{s:intro}.  
We next describe its standard discrete-stage reformulation (see, for example, (\cite{put94}, Ch.\ 11)).
If, at decision epoch $t_k$  the project lies in state $X_k = i$ and
 action $A_k = a$ is chosen, 
the joint distribution of the length $t_{k+1} - t_k$ of the following
\emph{$(i, a)$-stage} and embedded 
state $X_{k+1}$  is characterized by the transition distribution function
\[
H_{ij}^a(t) \triangleq 
\Prob_i^a\left\{t_{k+1} - t_k \leqslant t, X_{k+1} = j\right\}, \quad t \geqslant 0, 
\]
having Laplace--Stieltjes transform (LST)
\[
\psi_{ij}^{a}(\alpha) \triangleq \int_0^\infty e^{-\alpha t} \, dH_{ij}^a(t) = 
\Ex_i^a\big[1_{\{X_{k+1} = j\}} e^{-\alpha (t_{k+1} - t_k)}\big], \quad \alpha > 0,
\]
where $\Prob_i^a$ and $\Ex_i^a$ denote probability and expectation conditional on starting from state $i$ with action $a$ ($X_k = i, A_k = a$).

From $H_{ij}^a(t)$, we have that the distribution of the length of an
$(i, a)$-stage is
\[
H_i^a(t) \triangleq \Prob_i^a\{t_{k+1} - t_k \leqslant t\} = \sum_{j \in \mathcal{N}} H_{ij}^a(t),
\]
having LST
\begin{equation}
\label{eq:phiialst}
\psi_{i}^{a}(\alpha) \triangleq 
\Ex_i^a[e^{-\alpha (t_{k+1} - t_k)}]
= \sum_{j \in \mathcal{N}} \psi_{ij}^{a}(\alpha),
\end{equation}
and mean
\[
m_i^a \triangleq \Ex_i^a[t_{k+1} - t_k \mid X_k = i, A_k = a]
= \int_0^\infty t \, dH_i^a(t).
\]

The one-stage transition probabilities for the
embedded process are
\[
p_{ij}^a \triangleq 
\Prob_i^a\{X_{k+1} = j  \mid  X_k = i, A_k = a\} = 
\lim_{t \to \infty} H_{ij}^a(t) = 
\lim_{\alpha \searrow 0}  \psi_{ij}^{a}(\alpha).
\]

Recall that the process $X(t)$ may
change state between successive decision epochs. 
Its dynamics \emph{within} an
$(i, a)$-stage are characterized by 
\[
\tilde{p}_{ij}^{a}(s) \triangleq \Prob_i^a\{X(t_{k} + s) = j \, \mid \, t_{k+1} - t_k > s\}, \quad s \geqslant 0, 
\]
the conditional probability that, $s$ time
  units after the start of an $(i, a)$-stage, and~given that this is still ongoing, the~project occupies state $j$.
We can thus formulate the expected discounted amount of resource consumed and 
  reward obtained in an $(i, a)$-stage, respectively, as~\begin{equation}
\label{tildetia}
q_i^a \triangleq \Ex_i^a\Big[\int_{t_k}^{t_{k+1}} Q_{X(t)}^{A_k} e^{-\alpha (t-t_k)}
   \, dt\Big] = 
 \sum_{j \in \mathcal{N}} Q_{j}^{a} \int_0^\infty \tilde{p}_{ij}^{a}(s) \bar{H}_i^a(s) e^{-\alpha
   s} \, ds
\end{equation}
and
\begin{equation}
\label{tilderia}
r_i^a \triangleq \Ex_i^a\Big[\int_{t_k}^{t_{k+1}} R_{X(t)}^{A_k} e^{-\alpha (t-t_k)}
   \, dt\Big] = 
 \sum_{j \in \mathcal{N}} R_{j}^{a} \int_0^\infty \tilde{p}_{ij}^{a}(s) \bar{H}_i^a(s) e^{-\alpha
   s} \, ds,
\end{equation}
where $\bar{H}_i^a(s) \triangleq 1 - H_i^a(s)$ is the tail distribution of the length of an $(i, a)$-stage.

Recall that we denote by $n \triangleq |\mathcal{N}|$ the number of project~states.

\subsection{Indexability, Whittle Index, and~the Achievable Resource--Reward Performance~Region}
\label{s:wrm}
We start by considering the
   discounted criterion with rate $\alpha > 0$,
deferring discussion of the average criterion to Section~\ref{s:rac}.
We consider the following project performance metrics to evaluate a policy $\pi \in \Pi$ starting from a state $i$: the \emph{reward metric}
\begin{equation}
\label{eq:fipid}
F_i^\pi 
\triangleq \Ex_i^\pi\Big[\int_{0}^\infty R_{X(t)}^{A(t)} 
  e^{-\alpha t} \, dt\Big] = \Ex_i^\pi\bigg[\sum_{k=0}^\infty r_{X_k}^{A_k}
  e^{-\alpha t_k}\bigg],
\end{equation}
measuring the expected discounted value of rewards obtained, and~the 
\emph{resource metric}
\begin{equation}
\label{eq:gipid}
G_i^\pi 
\triangleq \Ex_i^\pi\Big[\int_{0}^\infty Q_{X(t)}^{A(t)} 
  e^{-\alpha t} \, dt\Big] = \Ex_i^\pi\bigg[\sum_{k=0}^\infty
  q_{X_k}^{A_k} e^{-\alpha t_k}\bigg],
\end{equation}
giving the expected discounted quantity of resource consumed. 
Note that the right-hand side expectations  in
(\ref{eq:fipid})--(\ref{eq:gipid}) are formulated in terms of the aforementioned
discrete-stage~reformulation.

We will also refer to the 
 metrics
obtained by drawing the initial state from a probability mass vector
 $\mathbf{p} = (p_i)_{i \in \mathcal{N}}$, given by
\begin{equation}
\label{eq:fpigpip}
F_{\mathbf{p}}^\pi \triangleq \sum_{i \in \mathcal{N}} p_i F_i^\pi \quad \text{ and } \quad
G_{\mathbf{p}}^\pi \triangleq \sum_{i \in \mathcal{N}} p_i G_i^\pi.
\end{equation}

To calibrate the marginal value of engaging the project in each
state, we
introduce a parameter $\lambda \in \mathbb{R}$ representing the \emph{resource $($unit$)$ price}, 
and define the \emph{$($net$)$ project value  metric} $V_{\mathbf{p}}^\pi(\lambda) \triangleq F_{\mathbf{p}}^\pi - \lambda G_{\mathbf{p}}^\pi$, so the \emph{optimal project value} is $V_{\mathbf{p}}^*(\lambda) \triangleq \sup_{\pi \in \Pi} \, V_{\mathbf{p}}^\pi(\lambda)$. We also write $V_{i}^\pi(\lambda)$ and $V_{i}^*(\lambda)$ as in Section~\ref{s:intro}.
Consider the parametric family $\mathcal{P}$ of 
\emph{$\lambda$-price problems} $(P_\lambda)$ defined in (\ref{eq:nuwpro}) for $\lambda \in \mathbb{R}$.
Thus, for~given $\lambda$, problem (\ref{eq:nuwpro}) is to  find an admissible policy that maximizes the project~value.

The \emph{dynamic programming} (DP) optimality equations
for $\lambda$-price problem $(P_\lambda)$ are
\begin{equation}
\label{eq:dpequ}
V_i^*(\lambda) = \max_{a \in \{0, 1\}} \, r_i^a - \lambda q_i^a + 
  \sum_{j \in \mathcal{N}} \psi_{ij}^a V_j^*(\lambda), \quad i \in \mathcal{N},
\end{equation}
where we write $\psi_{ij}^a = \psi_{ij}^a(\alpha)$.
Classic results of the SMDP theory ensure existence of a \emph{$P_\lambda$-optimal
policy} $\pi_\lambda^*$ solving (\ref{eq:nuwpro}) that is stationary deterministic ($\pi_\lambda^* \in \Pi^{\textup{SD}}$).

We  will also consider the optimal project value starting from state $i$ with initial action $a$, given by
\begin{equation}
\label{eq:vlastarx}
V_i^{\langle a, *\rangle}(\lambda) \triangleq
r_i^a - \lambda q_i^a + 
  \sum_{j \in \mathcal{N}} \psi_{ij}^a V_j^*(\lambda),
\end{equation}
and say that \emph{action $a$ is $P_\lambda$-optimal in state $i$}  if  $V_i^{\langle a, *\rangle}(\lambda) \geqslant V_i^{\langle 1-a, *\rangle}(\lambda)$.
It will be convenient to reformulate such a definition in terms of the sign of $\Delta_{a=0}^{a=1} V_i^{\langle a, *\rangle}(\lambda) \triangleq V_i^{\langle 1, *\rangle}(\lambda) - V_i^{\langle 0, *\rangle}(\lambda)$,  the~\emph{marginal value of engaging the project in state $i$}.  Thus,  
action $a = 1$ is $P_\lambda$-optimal in $i$ if $\Delta_{a=0}^{a=1} V_i^{\langle a, *\rangle}(\lambda) \geqslant 0$; $a = 0$ is $P_\lambda$-optimal in $i$ if $\Delta_{a=0}^{a=1} V_i^{\langle a, *\rangle}(\lambda) \leqslant 0$; and, hence, both actions are $P_\lambda$-optimal in $i$ if, and only if,
\begin{equation}
\label{eq:bothactopt}
\Delta_{a=0}^{a=1} V_i^{\langle a, *\rangle}(\lambda) = 0.
\end{equation}

Since actions are binary, 
it is convenient to represent a stationary deterministic policy by its \emph{active set} $S
\subseteq \mathcal{N}$,  which is the subset of states
where,  at~a decision
epoch, the~policy chooses the active action. We will refer to the
\emph{$S$-active policy} and write $F_i^S$, $G_i^S$, and~$V_{i}^S(\lambda)$.

Thus, we can  reformulate $\lambda$-price
problem (\ref{eq:nuwpro}) as the \emph{discrete optimization problem}
\begin{equation}
\label{eq:nuwprodo}
(P_\lambda)\colon \quad \textup{find} \enspace S^* \subseteq \mathcal{N} \enspace \textup{such that} \enspace V_{i}^{S^*}(\lambda)  = V_{i}^*(\lambda) \enspace \textup{for every} \enspace i \in \mathcal{N}.
\end{equation}

For a given resource price $\lambda$, the~\emph{$P_\lambda$-optimal active} and \emph{passive sets} are given by
\begin{equation}
\label{eq:Sstarlam}
S^{*, 1}(\lambda) \triangleq \big\{i \in \mathcal{N}\colon \Delta_{a=0}^{a=1} V_i^{\langle a, *\rangle}(\lambda) \geqslant 0\big\} \quad \textup{and} \quad S^{*, 0}_\lambda \triangleq \big\{i \in \mathcal{N}\colon  \Delta_{a=0}^{a=1} V_i^{\langle a, *\rangle}(\lambda) \leqslant 0\big\}.
\end{equation}

\subsection{Indexability}
\label{s:piip}
We will address the parametric problem collection $\mathcal{P}$ in (\ref{eq:nuwpro}) through
 the concept of
\emph{indexability}, extended in~\cite{whit88} from classic to restless bandits with resource consumption  $Q_i^a \triangleq a$, and~further extended in~\cite{nmmp02} to general $Q_i^a$. 

Under indexability, the~$P_\lambda$-optimal active and passive sets 
$S^{*, 1}(\lambda)$ and  $S^{*, 0}(\lambda)$ are characterized by an 
\emph{index} attached to project states. Note that the definition below refers to the \emph{sign} function $\sgn\colon \mathbb{R} \to \{-1, 0, 1\}$, and~follows the formulation of indexability in (\cite{nmmor20}, Definition 1). 

\begin{definition}[Indexability and Whittle index]
\label{def:indxnm}
The project is \emph{indexable} if there exists a mapping $i \mapsto \lambda_i^*$ for $i \in  \mathcal{N}$, such that
\begin{equation}
\label{eq:indxsigndef}
\sgn \Delta_{a=0}^{a=1} V_i^{\langle a, *\rangle}(\lambda) = \sgn (\lambda_i^* - \lambda), \quad \textup{for every state} \enspace i \in \mathcal{N} \enspace \textup{and price} \enspace \lambda \in \mathbb{R}.
\end{equation}
We call $\lambda_i^*$ the project's $($generalized$)$ \emph{Whittle index} $($\emph{Gittins index} if the project is non-restless$).$
\end{definition}
Thus, under~indexability, the~$P_\lambda$-optimal active and passive sets 
$S^{*, 1}(\lambda)$ and  $S^{*, 0}(\lambda)$ are characterized as
\begin{equation}
\label{eq:Sstarlam2}
S^{*, 1}(\lambda) \triangleq \big\{i \in \mathcal{N}\colon \lambda_i^* \geqslant \lambda\big\} \quad \textup{and} \quad S^{*, 0}_\lambda \triangleq \big\{i \in \mathcal{N}\colon  \lambda_i^* \leqslant \lambda\big\}.
\end{equation}

As for the economic interpretation of the Whittle index,
it is shown in~\cite{nmmp02,nmmor06,nmtop07}
that $\lambda_i^*$ measures the \emph{marginal productivity of the resource
at state $i$}. 

\subsection{PCL-Indexability and Adaptive-Greedy~Algorithm}
\label{s:pclic}
In applications of Whittle indexation to given restless bandit models, researchers are
concerned with
establishing analytically their indexability and efficiently computing the index.
For that purpose, the~author introduced the \emph{PCL-indexability
  conditions} in~\cite{nmaap01,nmmp02,nmmor06} because they are grounded on the satisfaction of
\emph{partial conservation laws} (PCLs). See~\cite{nmconsLawsEORMS} for a survey on conservation laws in stochastic scheduling and multi-class queueing~systems.

Such conditions can be used to establish indexability relative to a nonempty
family $\mathcal{F} \subseteq 2^{\mathcal{N}}$ of active sets that one should posit a
priori (based on insight on the model's structure).

\begin{definition}[$\mathcal{F}$-indexability]
\label{def:findx}
{We say that the project, or~more precisely, the parametric collection $\mathcal{P}$ of $\lambda$-price problems 
$(P_\lambda)$ in (\ref{eq:nuwprodo}), is
\emph{$\mathcal{F}$-indexable} if: (i) it is indexable; and (ii) for every price $\lambda$, there exists an optimal active set $S^*(\lambda) \in \mathcal{F}$, such that the $S^*(\lambda)$-active policy is $P_\lambda$-optimal.}
\end{definition}

Thus, for~an $\mathcal{F}$-indexable project, 
problem (\ref{eq:nuwprodo}) can be reduced to
\begin{equation}
\label{eq:nuwpf}
(P_\lambda)\colon \quad \textup{find} \enspace S^* \in \mathcal{F} \enspace \textup{such that} \enspace V_{i}^{S^*}(\lambda)  = V_{i}^*(\lambda) \enspace \textup{for every} \enspace i \in \mathcal{N},
\end{equation}
since, in such a case \emph{$\mathcal{F}$-policies} (those with active sets $S \in
\mathcal{F}$) are optimal for (\ref{eq:nuwprodo}).

Algorithmic considerations lead us to require that active-set family
$\mathcal{F}$ satisfies the 
connectedness conditions stated next.
We will
refer to the \emph{inner boundary of $S$ with respect to $\mathcal{F}$}, given by 
\[
\partial_{\mathcal{F}}^{\textup{in}} S \triangleq \{i \in S\colon
S \setminus \{i\} \in \mathcal{F}\}, 
\]
and to the \emph{outer boundary of $S$ with respect to $\mathcal{F}$}, given by 
\[
\partial_{\mathcal{F}}^{\textup{out}} S \triangleq \{j \in S^c\colon
S \cup \{j\} \in \mathcal{F}\}.
\]

\begin{assumption}
\label{ass:ccpcl}
\begin{itemize}[align=parleft,leftmargin=*,labelsep=5.5mm]
\item[\textup{(i)}] For every nonempty $S \in \mathcal{F}$, $\partial_{\mathcal{F}}^{\textup{in}} S$ is nonempty.
\item[\textup{(ii)}] For every $S \in \mathcal{F} \setminus \{\mathcal{N}\}$, $\partial_{\mathcal{F}}^{\textup{out}} S$
is nonempty.
\end{itemize} 
\end{assumption}

Note that from the requirement that $\mathcal{F}$ be nonempty, along with
Assumption \ref{ass:ccpcl}, it follows that $\emptyset, \mathcal{N} \in \mathcal{F}$.

To formulate the following result, we need to consider certain
\emph{marginal metrics}.
For an action $a \in \{0, 1\}$ and an active set $S \subseteq \mathcal{N}$, 
denote by $\langle a, S\rangle$ the policy taking initially action
$a$, and~thereafter following the 
$S$-active policy. For~a state $i$ and an active set $S$, 
define the  \emph{marginal resource (consumption) metric} by
\begin{equation}
\label{eq:mwm}
g_i^S \triangleq G_i^{\langle 1, S\rangle} - G_i^{\langle 0, S\rangle},
\end{equation}
that is, as~the marginal increase in resource consumed resulting
from taking first the active rather than the passive action
at state $i$, provided that the $S$-active policy is followed~thereafter.

Define also the \emph{marginal reward metric} by
\begin{equation}
\label{eq:mrm}
f_i^S \triangleq F_i^{\langle 1, S\rangle} - F_i^{\langle 0, S\rangle},
\end{equation}
that is, as~the marginal increase in the value of rewards gained.
Finally, for~$g_i^S \neq 0$, 
define the \emph{marginal productivity rate metric}  by
\begin{equation}
\label{eq:mpm}
\lambda_i^S \triangleq \frac{f_i^S}{g_i^S}.
\end{equation}

The following definition refers to the \emph{adaptive-greedy index
  algorithm}, which is given in
\mbox{Algorithms~\ref{fig:agaf} and \ref{fig:agafbu}} in its top-down and bottom-up
versions, respectively. In~the former, index values are computed from
largest to smallest, whereas in the latter they are computed in the
opposite~order.

\begin{algorithm}[H]
\caption{\hl{Adaptive-greedy algorithm: top-down version} %Please confirm whether this is an algorithm, if it is, please change it to the format of the algorithm
% Answer: Yes, this is an algorithm, I have changed the environment and the references to it accordingly
$\mathrm{AG}_{\mathcal{F}}^{\textup{TD}}$.}
\begin{center}
\fbox{%
\begin{minipage}{1.8in}
\textbf{Output:}
$\{j_k, \lambda^*_{j_k}\}_{k=1}^{n}$
\begin{tabbing}
$S_0 := \emptyset$ \\
\textbf{for} \= $k := 1$ \textbf{to}  $n$ \textbf{do} \\
\> \textbf{pick}  
 $j_k \in \argmax
      \{\lambda^{S_{k-1}}_{j}\colon
                 j \in \partial^{\textup{out}}_{\mathcal{F}} S_{k-1}\}$ \\
 \>  $\lambda^*_{j_k} := 
 \lambda^{S_{k-1}}_{j_k}$;  \, $S_{k} := S_{k-1} \cup \{j_k\}$ \\
\textbf{end} \{ for \}
\end{tabbing}
\end{minipage}}
\end{center}
\label{fig:agaf}
\end{algorithm}
\unskip

\begin{algorithm}[H]
\caption{\hl{Adaptive-greedy algorithm: bottom-up version }  %Please confirm whether this is an algorithm, if it is, please change it to the format of the algorithm
% Answer: Yes, this is an algorithm, I have changed the environment and the references to it accordingly
$\mathrm{AG}_{\mathcal{F}}^{\textup{BU}}$.}
\begin{center}
\fbox{%
\begin{minipage}{1.8in}
\textbf{Output:}
$\{i_k, \lambda^*_{i_k}\}_{k=1}^{n}$
\begin{tabbing}
$S_0' := \mathcal{N}$ \\
\textbf{for} \= $k := 1$ \textbf{to}  $n$ \textbf{do} \\
\> \textbf{pick}  
 $i_k \in \argmin
      \{\lambda^{S_{k-1}'}_{i}\colon
                 i \in \partial^{\textup{in}}_{\mathcal{F}} S_{k-1}'\}$ \\
 \>  $\lambda^*_{i_k} := 
 \lambda^{S_{k-1}'}_{i_k}$;  \, $S_{k}' := S_{k-1}' \setminus \{i_k\}$ \\
\textbf{end} \{ for \}
\end{tabbing}
\end{minipage}}
\end{center}
\label{fig:agafbu}
\end{algorithm}

\begin{definition}[PCL$(\mathcal{F})$-indexability]
\label{def:pcli} 
{We say that the project is \emph{PCL$(\mathcal{F})$-indexable}
  if:
\begin{itemize}[align=parleft,leftmargin=*,labelsep=5.5mm]
\item[(i)]
for every active set $S \in \mathcal{F}$,  
$g_i^S > 0$ for each $i \in \mathcal{N}$; and
\item[(ii)] algorithm $\mathrm{AG}_{\mathcal{F}}^{\textup{TD}}$
  computes a monotone non-increasing index sequence $(\lambda_{j_1}^* \geqslant \cdots \geqslant \lambda_{j_n}^*)$; or
  algorithm $\mathrm{AG}_{\mathcal{F}}^{\textup{BU}}$ computes a
  monotone non-decreasing
  index sequence $(\lambda_{i_1}^* \leqslant \cdots \leqslant \lambda_{i_n}^*)$.
\end{itemize}
}
\end{definition}

\begin{theorem}
\label{the:pclind}
Suppose that the project is PCL$(\mathcal{F})$-indexable$.$ Then, it is $\mathcal{F}$-indexable and
the index $\lambda_i^*$ computed by either adaptive-greedy index algorithm
is its Whittle index$.$
\end{theorem}

Theorem \ref{the:pclind} was introduced and proven, in~increasingly
general settings, in~(\cite{nmaap01}, Corollary 2), 
\mbox{(Ref. \cite{nmmp02}, Theorem 6.3)} and (\cite{nmmor06}, Theorem 4.1).
The author's work reviewed in~\cite{nmtop07} demonstrates the
applicability of such a result to a number of relevant
restless bandit~models.

Regarding the interpretation of positive marginal resource condition (i) in Definition
\ref{def:pcli}, it is shown in (\cite{nmmp02}, Proposition 6.2) and 
in (\cite{nmmor06},  Lemma 4.3) that it can be
reformulated in terms of the resource metric as follows:  for $\mathbf{p} > \mathbf{0}$,
\begin{equation}
\label{eq:mwpreform}
G_{\mathbf{p}}^{S \setminus \{i\}} < G_{\mathbf{p}}^S < G_{\mathbf{p}}^{S \cup \{j\}}, \quad 
S \in \mathcal{F}, i \in S, j \in S^c.
\end{equation}
Thus, condition (i) represents a
strong form of monotone increasingness of resource metric $G_{\mathbf{p}}^S$ in its active set
$S$ relative to inclusion.
Additionally, recalling the definition of $G_{\mathbf{p}}^\pi$ in (\ref{eq:fpigpip}) in
terms of metrics \mbox{$G_i^\pi$, 
(\ref{eq:mwpreform})  can},  in turn, be reformulated in terms of the latter metrics
as follows: for $S \in \mathcal{F}$, \mbox{$i \in S$ and $j \in S^c$},
\begin{equation}
\label{eq:gspclmon2}
\mathbf{G}^{S \setminus \{i\}} \lneq \mathbf{G}^S \lneq \mathbf{G}^{S
  \cup \{j\}}, \quad S \in \mathcal{F}, i \in S, j \in S^c
\end{equation}
where $\mathbf{G}^\pi = (G_{i_0}^\pi)_{i_0 \in \mathcal{N}}$ and
``$\lneq$'' means ``less than or equal to componentwise, but~not~equal.''

Along with condition (i), condition (ii) in Definition \ref{def:pcli} is interpreted in (\cite{nmmp02}, Section~6.4) in terms of
satisfaction of the \emph{law of diminishing marginal
returns} (to resource usage) in  economics, relative to $\mathcal{F}$-policies.
Such an interpretation is based on the result in (\cite{nmmp02}, Proposition
6.4(a)) that the marginal
productivity rates $\lambda_i^S$ in (\ref{eq:mpm}) for $S \in \mathcal{F}$
can be recast in terms of resource and reward metrics as
\begin{equation}
\label{eq:nuisref}
\lambda_i^S = 
\begin{cases}
\displaystyle \frac{F_{\mathbf{p}}^{S} - F_{\mathbf{p}}^{S \setminus \{i\}}}{G_{\mathbf{p}}^{S} - G_{\mathbf{p}}^{S \setminus \{i\}}}, \quad & i \in
S \\ \\
\displaystyle \frac{F_{\mathbf{p}}^{S \cup \{i\}} - F_{\mathbf{p}}^S}{G_{\mathbf{p}}^{S \cup \{i\}} - G_{\mathbf{p}}^S}, \quad & i \in
S^c.
\end{cases}
\end{equation}

Such a reformulation allows us, in turn, to reformulate the adaptive-greedy
algorithms above in a geometrically intuitive form in the resource--reward
$(\gamma, \phi)$ plane,
as shown in Algorithms~\ref{fig:agafref} and \ref{fig:agafburef}.
We thus see that the top-down algorithm
$\mathrm{AG}_{\mathcal{F}}^{\textup{TD}}$ aims to traverse the upper
boundary $\bar{\partial} \mathcal{R}_{\mathbf{p}}$ of the achievable resource--reward
performance region $\mathcal{R}_{\mathbf{p}}$ from left to right using only active
sets in $\mathcal{F}$, whereas the
bottom-up version $\mathrm{AG}_{\mathcal{F}}^{\textup{BU}}$ seeks to
traverse such an upper
boundary from right to left,  proceeding in a greedy
fashion at each step.
In such a setting, condition (ii) in Definition \ref{def:pcli} means
that the function obtained by linear interpolation on the points $(G_{\mathbf{p}}^S, F_{\mathbf{p}}^S)$ produced by either
algorithm is \emph{concave}.
The remarkable result in Theorem \ref{the:pclind} is that this, along
with condition (i),
suffices to ensure that such a function characterizes the upper
boundary $\bar{\partial} \mathcal{R}_{\mathbf{p}}$ and the Whittle~index.

In~\cite{nmmor06}, such results are extended to restless projects
on the denumerable state space of nonnegative integers, for~which the resource metric is increasing along the
 nested family of active sets induced by the corresponding ordering.
Further, in~Section~3 of that paper the $\mathcal{F}$-indexability
of such projects is characterized in terms of
satisfaction of the law of diminishing marginal returns
relative to $\mathcal{F}$-policies.
In~\cite{nmmor20}, such results are further extended to continuous-state~projects.

\begin{algorithm}[H]
\caption{\hl{Reformulated adaptive-greedy index algorithm: top-down version}  %Please confirm whether this is an algorithm, if it is, please change it to the format of the algorithm
% Answer: Yes, this is an algorithm, I have changed the environment and the references to it accordingly
$\mathrm{AG}_{\mathcal{F}}^{\textup{TD}}$.}
\begin{center}
\fbox{%
\begin{minipage}{1.8in}
\textbf{Output:}
$\{j_k, \lambda^*_{j_k}\}_{k=1}^{n}$
\begin{tabbing}
$S_0 := \emptyset$ \\
\textbf{for} \= $k := 1$ \textbf{to}  $n$ \textbf{do} \\
\> \textbf{pick}  
 $\displaystyle j_k \in \argmax
      \left\{\frac{F_{\mathbf{p}}^{S_{k-1} \cup \{j\}} - F_{\mathbf{p}}^{S_{k-1}}}{G_{\mathbf{p}}^{S_{k-1} \cup \{j\}} - G_{\mathbf{p}}^{S_{k-1}}}\colon
                 j \in \partial^{\textup{out}}_{\mathcal{F}} S_{k-1}\right\}$ \\
 \>  $\lambda^*_{j_k} := 
 \lambda^{S_{k-1}}_{j_k}$;  \, $S_{k} := S_{k-1} \cup \{j_k\}$ \\
\textbf{end} \{ for \}
\end{tabbing}
\end{minipage}}
\end{center}
\label{fig:agafref}
\end{algorithm}
\unskip

\begin{algorithm}[H]
\caption{\hl{Reformulated adaptive-greedy index algorithm: bottom-up version}  %Please confirm whether this is an algorithm, if it is, please change it to the format of the algorithm
% Answer: Yes, this is an algorithm, I have changed the environment and the references to it accordingly
$\mathrm{AG}_{\mathcal{F}}^{\textup{BU}}$.}
\begin{center}
\fbox{%
\begin{minipage}{1.8in}
\textbf{Output:}
$\{i_k, \lambda^*_{i_k}\}_{k=1}^{n}$
\begin{tabbing}
$S_0' := \mathcal{N}$ \\
\textbf{for} \= $k := 1$ \textbf{to}  $n$ \textbf{do} \\
\> \textbf{pick}  
 $\displaystyle i_k \in \argmin
      \left\{\frac{F_{\mathbf{p}}^{S_{k-1}'} - F_{\mathbf{p}}^{S_{k-1}' \setminus
          \{i\}}}{G_{\mathbf{p}}^{S_{k-1}'} - G_{\mathbf{p}}^{S_{k-1}' \setminus \{i\}}}\colon
                 i \in \partial^{\textup{in}}_{\mathcal{F}} S_{k-1}'\right\}$ \\
 \>  $\lambda^*_{i_k} := 
 \lambda^{S_{k-1}'}_{i_k}$;  \, $S_{k}' := S_{k-1}' \setminus \{i_k\}$ \\
\textbf{end} \{ for \}
\end{tabbing}
\end{minipage}}
\end{center}
\label{fig:agafburef}
\end{algorithm}
\unskip

\section{Laying the Groundwork for an Efficient Implementation of the Adaptive-Greedy~Algorithm}
\label{s:ecge}
This section lays the groundwork for developing an efficient implementation of the adaptive-greedy index algorithm.
It draws on LP methods, based on formulating the $\lambda$-price \mbox{problem 
(\ref{eq:nuwpro})} as a parametric LP problem,  and~elucidating
the structure of its simplex~tableaux.

\subsection{Optimality Equations and Parametric LP~Formulation}
\label{s:plpf}
While the LP formulation below is well-known in SMDP theory (cf.\ \cite{kallenb02}), for~convenience
we  next  outline it, starting from 
 the optimality {\hl{equations}}~(\ref{eq:dpequ}) for (\ref{eq:nuwpro}).
 % I changed the ``Equations'' by ``equations'', which I think is the proper form
The primal LP formulation of such optimality equations is
\begin{align*}
V_{\mathbf{p}}^*(\lambda) = & \min \, \sum_{j \in \mathcal{N}} p_j~v_j \\
    & \text{subject to} \\
    x_i^a\colon 
& v_i - \sum_{j \in \mathcal{N}} \psi_{ij}^a v_j \geqslant r_i^a - \lambda
q_i^a, \quad (i, a) \in \mathcal{N} \times \{0, 1\}.
\end{align*}

Our analyses will instead use the dual problem, 
\begin{align*}
V_{\mathbf{p}}^*(\lambda) = & \max \, \sum_{(j, a) \in \mathcal{N} \times \{0, 1\}} (r_j^a - \lambda q_j^a) x_j^a
    \\
    & \text{subject to} \\
v_j\colon &
\sum_{a \in \{0, 1\}}
(x_j^a -   \sum_{j \in \mathcal{N}} \psi_{ij}^a x_i^a)
 = p_j, \quad j \in \mathcal{N} \\
& x_j^a \geqslant 0,  \quad (j, a) \in \mathcal{N} \times \{0, 1\}.
\end{align*}
It will be convenient to deal with the latter in matrix notation.
We treat occupation-measure vectors $\mathbf{x}^a$ and probability vectors $\mathbf{p}$ as column vectors, and performance/coefficient vectors (e.g., $\mathbf{r}^a$, $\mathbf{q}^a$, $\mathbf{F}^S$, $\mathbf{G}^S$) as row vectors.
 We can thus write
 \begin{equation}
\label{eq:nuwplp}
\begin{split}
V_{\mathbf{p}}^*(\lambda) =
 & \max \,  (\mathbf{r}^0 - \lambda \mathbf{q}^0) \mathbf{x}^{0} 
 + 
(\mathbf{r}^1 - \lambda \mathbf{q}^1) \mathbf{x}^{1} \\
& \text{subject to} \\
& 
\begin{bmatrix} 
\transp{(\mathbf{I} -  \boldsymbol{\Psi}^0)} &
\transp{(\mathbf{I} -  \boldsymbol{\Psi}^1)} 
\end{bmatrix} 
\begin{pmatrix} \mathbf{x}^{0} \\  \mathbf{x}^{1} \end{pmatrix}
= \mathbf{p} \\
& \mathbf{x}^{0}, \mathbf{x}^{1} 
 \geqslant \mathbf{0}.
\end{split}
\end{equation}

Dual variables $x_j^a$  correspond to 
the project's \emph{discounted state-action occupation
  measures}. 
For an initial
state $i$, policy $\pi$, action $a$
 and state $j$, let
\[
x_{ij}^{a, \pi} \triangleq \Ex_i^\pi\left[\sum_{k=0}^\infty 
 1_{\{X(t_k) = j, A(t_k) = a\}} e^{-\alpha t_k}\right]
\]
be the expected discounted number of $(j, a)$-stages
under $\pi$ starting
from $i$.
If the initial state is drawn from $\mathbf{p}$,  $x_j^a$
corresponds to occupation measure {\hl{$x_{\mathbf{p} j}^{a, \pi} \triangleq \sum_i p_i
x_{ij}^{a, \pi}$}}. 
% Here I changed x_j^{a, \pi} to x_{\mathbf{p} j}^{a, \pi} for correctness
Note that
 reward and resource metrics are linear functions of such occupation measures: writing
 {\hl{$\mathbf{x}_{\mathbf{p}}^{a, \pi} = (x_{\mathbf{p}  j}^{a, \pi})_{j \in \mathcal{N}}$}},
 % Here  I changed $\mathbf{x}^{a, \pi} = (x_{j}^{a, \pi})$ to $\mathbf{x}_{\mathbf{p}}^{a, \pi} = (x_{\mathbf{p}  j}^{a, \pi})_{j \in \mathcal{N}}$ for correctness
\begin{equation}
\label{eq:figilf}
F_{\mathbf{p}}^\pi = \sum_{(j, a) \in \{0, 1\} \times \mathcal{N}} r_j^a x_{\mathbf{p}  j}^{a, \pi} = 
  \mathbf{r}^0 \mathbf{x}_{\mathbf{p}}^{0, \pi} + 
  \mathbf{r}^1 \mathbf{x}_{\mathbf{p}}^{1, \pi} \enspace \textup{and} \enspace
G_{\mathbf{p}}^\pi  = \sum_{(j, a) \in \{0, 1\} \times \mathcal{N}} q_j^a x_{\mathbf{p}  j}^{a, \pi} = 
  \mathbf{q}^0 \mathbf{x}_{\mathbf{p}}^{0, \pi} + 
  \mathbf{q}^1 \mathbf{x}_{\mathbf{p}}^{1, \pi}.
\end{equation}
% Above, in equation {eq:figilf} I changed for correctness the x_j^{a, \pi} to x_{\mathbf{p} j}^{a, \pi}, and the \mathbf{x}^{a, \pi}  to \mathbf{x}_{\mathbf{p}}^{a, \pi}

\subsection{Bases, Basic Feasible Solutions, and~Reduced~Costs}
\label{s:bfsrc}
We next  analyze  the LP problem (\ref{eq:nuwplp}), starting by elucidating its
 \emph{basic feasible solutions}  (BFS). 
Each BFS is obtained from a \emph{basis} corresponding to an active set 
$S \subseteq \mathcal{N}$, and~hence we will refer to the
\emph{$S$-active BFS}.
Yet, note that different $S$'s might yield the same BFS.
For given $S$, we decompose 
the above vectors and matrices as 
\[
\mathbf{x}^{a} = \begin{pmatrix} \mathbf{x}_{S}^{a} \\
  \mathbf{x}_{S^c}^{a} \end{pmatrix}, \quad
\mathbf{p} = \begin{pmatrix} \mathbf{p}_{S} \\  \mathbf{p}_{S^c}
  \end{pmatrix}, \quad
\boldsymbol{\Psi}^a = \begin{pmatrix} \boldsymbol{\Psi}_{SS}^a &
  \boldsymbol{\Psi}_{S S^c}^a \\
  \boldsymbol{\Psi}_{S^c S}^a & \boldsymbol{\Psi}_{S^c S^c}^a \end{pmatrix}, \quad
\mathbf{I} = \begin{pmatrix} \mathbf{I}_{S S} & \mathbf{0}_{S S^c} \\
  \mathbf{0}_{S^c S} & \mathbf{I}_{S^c S^c} \end{pmatrix}, 
\]
and introduce the matrices
\begin{equation}
\label{eq:matdef}
\begin{split}
\boldsymbol{\Psi}^S & \triangleq \begin{pmatrix} \boldsymbol{\Psi}_{S S}^1 &
  \boldsymbol{\Psi}_{S S^c}^1 \\ \boldsymbol{\Psi}_{S^c
  S}^0 & \boldsymbol{\Psi}_{S^c S^c}^0
  \end{pmatrix}, \quad
\boldsymbol{\Psi}^{S^c} \triangleq \begin{pmatrix} \boldsymbol{\Psi}_{SS}^0 &
  \boldsymbol{\Psi}_{S S^c}^0 \\ \boldsymbol{\Psi}_{S^c S}^1 & \boldsymbol{\Psi}_{S^c S^c}^1
  \end{pmatrix}, \\
\mathbf{B}^S & \triangleq \transp{(\mathbf{I} -  \boldsymbol{\Psi}^S)}, \quad
\mathbf{N}^S \triangleq \transp{(\mathbf{I} -  \boldsymbol{\Psi}^{S^c})}, \quad
\mathbf{H}^S \triangleq (\mathbf{B}^S)^{-1}, \quad
\mathbf{A}^S \triangleq \mathbf{H}^S \mathbf{N}^S.
\end{split}
\end{equation}
Note that $\boldsymbol{\Psi}^S$  is the transition-transform matrix
for the $S$-active  policy.
Additionally, $\mathbf{B}^S$ is the \emph{basis matrix} in LP problem
(\ref{eq:nuwplp})
for 
the $S$-active BFS, which has as \emph{basic variables} 
\[
\begin{pmatrix} \mathbf{x}_S^1 \\ \mathbf{x}_{S^c}^0 \end{pmatrix},
\]
and $\mathbf{N}^S$ is the non-basic matrix of LP problem
(\ref{eq:nuwplp}) corresponding to
\emph{non-basic variables} 
\[
\begin{pmatrix} \mathbf{x}_S^0 \\ \mathbf{x}_{S^c}^1 \end{pmatrix}.
\]
We can thus rearrange the constraints in (\ref{eq:nuwplp}),
decomposing them  as
\[
\mathbf{B}^S 
\begin{pmatrix} \mathbf{x}_{S}^{1} \\ \mathbf{x}_{S^c}^{0}
\end{pmatrix}
+ 
\mathbf{N}^S
\begin{pmatrix} \mathbf{x}_{S}^{0} \\ \mathbf{x}_{S^c}^{1}
\end{pmatrix} = 
\mathbf{p}.
\]

We next evaluate performance
metrics under the $S$-active policy.
{\hl{The notation $x_{\mathbf{p} j}^{a, S}$ below refers to occupation measure $x_{\mathbf{p} j}^{a, \pi}$ under such a policy.}}
% I changed above for correctness $x_j^{a, S}$ to $x_{\mathbf{p} j}^{a, S}$ and $x_j^{a, \pi}$ to $x_{\mathbf{p} j}^{a, \p}$
Further, $\mathbf{F}^S =(F_j^S)_{j \in \mathcal{N}}$, $\mathbf{G}^S = (G_j^S)_{j \in \mathcal{N}}$, 
$\mathbf{f}^S = (f_j^S)_{j \in \mathcal{N}}$ and  
$\mathbf{g}^S = (g_j^S)_{j \in \mathcal{N}}$ are row~vectors.

\begin{lemma}
\label{lma:pmchar}
\mbox{ }
\begin{itemize}[leftmargin=2.3em,labelsep=4mm]
\item[\textup{(a)}]
{\hl{$\begin{pmatrix} \mathbf{x}_{\mathbf{p}  S}^{0, S} \\ \mathbf{x}_{\mathbf{p}  S^c}^{1, S}
\end{pmatrix} = \mathbf{0}$}} \, and~\, {\hl{$\begin{pmatrix} \mathbf{x}_{\mathbf{p}  S}^{1, S} \\ \mathbf{x}_{\mathbf{p}  S^c}^{0, S}
\end{pmatrix} = \mathbf{H}^S \mathbf{p}$}}.
% I changed above for correctness \mathbf{x}_{S} to \mathbf{x}_{\mathbf{p} S} and \mathbf{x}_{S^c} to \mathbf{x}_{\mathbf{p} S^c}
\item[\textup{(b)}] $\mathbf{G}^S = 
\begin{pmatrix}
\mathbf{q}_S^1 & \mathbf{q}_{S^c}^0
\end{pmatrix}
\mathbf{H}^S$.
\item[\textup{(c)}] 
$\mathbf{F}^S = 
\begin{pmatrix}
\mathbf{r}_S^1 & \mathbf{r}_{S^c}^0
\end{pmatrix} 
\mathbf{H}^S$.
\item[\textup{(d)}] 
$\begin{pmatrix}\mathbf{g}_S^S &
- \mathbf{g}_{S^c}^S
\end{pmatrix} = 
\begin{pmatrix}
\mathbf{q}_S^1 & \mathbf{q}_{S^c}^0
\end{pmatrix} 
\mathbf{A}^S  
 - 
\begin{pmatrix} \mathbf{q}_S^0
 & \mathbf{q}_{S^c}^1
\end{pmatrix}$.
\item[\textup{(e)}] 
$\begin{pmatrix}
\mathbf{f}_S^S &
- \mathbf{f}_{S^c}^S
\end{pmatrix}
= 
\begin{pmatrix}
\mathbf{r}_S^1 & \mathbf{r}_{S^c}^0
\end{pmatrix}
\mathbf{A}^S 
 - 
\begin{pmatrix}  \mathbf{r}_S^0 & \mathbf{r}_{S^c}^1 \end{pmatrix}$.
\end{itemize}
\end{lemma}
\begin{proof}
(a) Set to zero non-basic variables:
{\hl{$\mathbf{x}_{\mathbf{p} S}^{0, S} =  \mathbf{0}$ and $\mathbf{x}_{\mathbf{p} S^c}^{1, S}
 = \mathbf{0}$}}. To~calculate the values of basic variables, note that
\[
\mathbf{B}^S \begin{pmatrix} \mathbf{x}_{\mathbf{p} S}^{1} \\ \mathbf{x}_{\mathbf{p} S^c}^{0}
\end{pmatrix}  = \mathbf{p}, \enspace \textup{and hence} \enspace
\begin{pmatrix} \mathbf{x}_{\mathbf{p} S}^{1, S} \\ \mathbf{x}_{\mathbf{p} S^c}^{0, S}
\end{pmatrix} = \mathbf{H}^S \mathbf{p}.
\]
% I changed above for correctness \mathbf{x}_{S} to \mathbf{x}_{\mathbf{p} S} and \mathbf{x}_{S^c} to \mathbf{x}_{\mathbf{p} S^c}
(b) \hl{Use part (a) with $\mathbf{p} = \mathbf{e}_j$ (the $j$th unit
 coordinate vector) to formulate resource metrics  as}
\[
G_j^S = \begin{pmatrix}
\mathbf{q}_S^1 & \mathbf{q}_{S^c}^0
\end{pmatrix} 
 \begin{pmatrix} \mathbf{x}_{j S}^{1, S} \\ \mathbf{x}_{j S^c}^{0, S}
\end{pmatrix} 
= \begin{pmatrix}
\mathbf{q}_S^1 & \mathbf{q}_{S^c}^0
\end{pmatrix} 
\mathbf{H}^S \mathbf{e}_j 
\Longrightarrow
\mathbf{G}^S = \begin{pmatrix}
\mathbf{q}_S^1 & \mathbf{q}_{S^c}^0
\end{pmatrix} 
\mathbf{H}^S.
\]
% I changed above for correctness \mathbf{x}_{S} to \mathbf{x}_{j} S} and \mathbf{x}_{S^c} to \mathbf{x}_{j S^c}
(c)  \hl{Proceed as in (b) to formulate reward metrics as}
\[
F_j^S = 
\begin{pmatrix}
\mathbf{r}_S^1 & \mathbf{r}_{S^c}^0
\end{pmatrix} 
 \begin{pmatrix} \mathbf{x}_{j S}^{1, S} \\ \mathbf{x}_{j S^c}^{0, S}
\end{pmatrix} 
= \begin{pmatrix}
\mathbf{r}_S^1 & \mathbf{r}_{S^c}^0
\end{pmatrix} 
\mathbf{H}^S \mathbf{e}_j 
\Longrightarrow
\mathbf{F}^S = \begin{pmatrix}
\mathbf{r}_S^1 & \mathbf{r}_{S^c}^0
\end{pmatrix} 
\mathbf{H}^S.
\]
% I changed above for correctness \mathbf{x}_{S} to \mathbf{x}_{j} S} and \mathbf{x}_{S^c} to \mathbf{x}_{j S^c}

(d) Represent
marginal resource metrics (cf.\ (\ref{eq:mwm})) as
\begin{equation}
\label{eq:wsc1}
\mathbf{g}_S^S = \mathbf{G}_S^S - \mathbf{q}_S^0 -
  \mathbf{G}^S \transp{(\boldsymbol{\Psi}^0_{S \mathcal{N}})}  \quad
  \text{ and } \quad
\mathbf{g}_{S^c}^S = \mathbf{q}_{S^c}^1 + \mathbf{G}^S
  \transp{(\boldsymbol{\Psi}_{S^c N}^1)}
 -  \mathbf{G}_{S^c}^S.
\end{equation}
Recast now the equalities in (\ref{eq:wsc1}), using (b), as~
\begin{align*}
\begin{pmatrix}
\mathbf{g}_S^S &
- \mathbf{g}_{S^c}^S
\end{pmatrix}
& = \mathbf{G}^S \mathbf{N}^S  - 
\begin{pmatrix} \mathbf{q}_S^0 & \mathbf{q}_{S^c}^1 \end{pmatrix}
= \begin{pmatrix}
\mathbf{q}_S^1 & \mathbf{q}_{S^c}^0
\end{pmatrix}
 \mathbf{H}^S \mathbf{N}^S
 - \begin{pmatrix} \mathbf{q}_S^0 & 
\mathbf{q}_{S^c}^1
\end{pmatrix}  
\\
& = 
\begin{pmatrix}
\mathbf{q}_S^1 & \mathbf{q}_{S^c}^0
\end{pmatrix}
\mathbf{A}^S  
 - \begin{pmatrix} \mathbf{q}_S^0
 & \mathbf{q}_{S^c}^1
\end{pmatrix}.
\end{align*}

(e) This part follows as part (d). 
\end{proof}

The following result gives the \emph{reduced costs} of the LP problem (\ref{eq:nuwplp})
in terms of the marginal resource and 
reward metrics in (\ref{eq:mwm}) and (\ref{eq:mrm}). 
It further  uses such a
result to represent such an LP problem's objective
in terms of  reduced~costs.

\begin{lemma}
\label{lma:rclp}
The reduced costs for
  non-basic variables corresponding to the 
$S$-active BFS for LP problem \textup{(\ref{eq:nuwplp})}
are given by
\begin{equation}
\label{eq:rcvplp}
\begin{pmatrix}
\mathbf{f}_S^S - \lambda \mathbf{g}_S^S &
- \mathbf{f}_{S^c}^S + \lambda \mathbf{g}_{S^c}^S
\end{pmatrix},
\end{equation}
and hence, the LP problem's objective can be
formulated as
\begin{equation}
\label{eq:rwdl}
\sum_{(j, a) \in \mathcal{N} \times \{0, 1\}} (r_j^a - \lambda q_j^a) x_j^a
= F_{\mathbf{p}}^S - \lambda G_{\mathbf{p}}^S - 
\sum_{j \in S} (f_j^S - \lambda g_j^S) x_j^0 +
 \sum_{j \in S^c} (f_j^S - \lambda g_j^S) x_j^1.
\end{equation}
\end{lemma}
\begin{proof}
The results follow 
from the well-known representation of reduced costs
in LP theory, as~given by Lemma \ref{lma:pmchar}(d,e),
along with the well-known representation of the LP problem's objective 
in terms of the value of the current BFS  and reduced costs. 
\end{proof}

The following result, which follows from Lemma
\ref{lma:rclp}, represents metrics
$G_{\mathbf{p}}^\pi$, $F_{\mathbf{p}}^\pi$ and objective $F_{\mathbf{p}}^{\pi} - \lambda G_{\mathbf{p}}^\pi$ 
in terms of their values under the $S$-active policy. These  \emph{decomposition
identities} were first obtained in (\cite{nmaap01}, Theorem 3) and 
(\cite{nmmp02}, Proposition 6.1) via ad~hoc~arguments.

\begin{lemma}
\label{lma:dlaws} For any policy $\pi \in \Pi$\textup{:}
\mbox{ }
\begin{itemize}[leftmargin=2.3em,labelsep=4mm]
\item[\textup{(a)}] \hl{$\displaystyle G_{\mathbf{p}}^\pi = G_{\mathbf{p}}^S - \sum_{j \in S} g_j^S x_{\mathbf{p} j}^{0, \pi} +
 \sum_{j \in S^c} g_j^S x_{\mathbf{p} j}^{1, \pi}$.}
\item[\textup{(b)}] \hl{$\displaystyle F_{\mathbf{p}}^\pi = F_{\mathbf{p}}^S - \sum_{j \in S} f_j^S x_{\mathbf{p} j}^{0,
 \pi} +
 \sum_{j \in S^c} f_j^S x_{\mathbf{p} j}^{1, \pi}$.}
\item[\textup{(c)}] \hl{$\displaystyle F_{\mathbf{p}}^{\pi} - \lambda G_{\mathbf{p}}^\pi 
= F_{\mathbf{p}}^S - \lambda G_{\mathbf{p}}^S - \sum_{j \in S} (f_j^S - \lambda g_j^S) x_{\mathbf{p} j}^{0, \pi} +
 \sum_{j \in S^c} (f_j^S - \lambda g_j^S) x_{\mathbf{p} j}^{1, \pi}$.}
\end{itemize}
\end{lemma}
% Above I changed for correctness x_{j} by x_{\mathbf{p} j}

The following result, first established in 
 (\cite{nmmp02},  Corollary 6.1), elucidates the relationship between 
resource and reward metrics and the corresponding marginal metrics. 
\begin{lemma}
\label{lma:marginterp}
\begin{itemize}[leftmargin=2.3em,labelsep=4mm]
\item[\textup{(a)}] For $j \in S^c$,
\hl{$G_{\mathbf{p}}^{S \cup \{j\}} = G_{\mathbf{p}}^S + g_j^S x_{\mathbf{p} j}^{1, S \cup \{j\}}$  \enspace and \enspace
$F_{\mathbf{p}}^{S \cup \{j\}} = F_{\mathbf{p}}^S + f_j^S x_{\mathbf{p} j}^{1, S \cup \{j\}}$.}
\item[\textup{(b)}] For $j \in S$,
\hl{$G_{\mathbf{p}}^{S \setminus \{j\}} = G_{\mathbf{p}}^S - g_j^S x_{\mathbf{p} j}^{1, S \setminus
  \{j\}}$ \enspace
and  \enspace $F_{\mathbf{p}}^{S \setminus \{j\}} = F_{\mathbf{p}}^S - f_j^S x_{\mathbf{p} j}^{1, S \setminus \{j\}}$.}
\end{itemize}
\end{lemma} % Above I changed for correctness x_{j} by x_{\mathbf{p} j}
\begin{proof}
To obtain (a) use $\pi = S \cup \{j\}$
 in Lemma \ref{lma:dlaws} (a). Additionally,~similarly with (b) with  $\pi = S \setminus \{j\}$ and Lemma \ref{lma:dlaws} (b).
\end{proof}

\section{A Fast-Pivoting Index Algorithm for PCL$(\mathcal{F})$-Indexable Projects}
\label{s:fpia}
This section develops an efficient implementation of the
adaptive-greedy index algorithm above, focusing on its top-down version, for~a project that
 is PCL$(\mathcal{F})$-indexable.

We start by noting that the index of such a project
can be computed by deploying in the LP formulation (\ref{eq:nuwplp})
of the $\lambda$-price problem the classic \emph{parametric-objective
simplex algorithm}  in~\cite{gassSaaty55}.
In the present setting, 
the 
\emph{parametric simplex tableau}  for the 
$S$-active BFS is shown in Table~\ref{tab:invtabl}.
This tableau has basic variables 
 $\mathbf{x}_S^1$ and $\mathbf{x}_{S^c}^0$ in rows and non-basic
 variables $\mathbf{x}_S^0$ and $\mathbf{x}_{S^c}^1$ in columns, and~further,
has two  rows of reduced-costs for non-basic variables.
The tableau does not include right-hand sides or objectives, as they are
not required in this context.
The tableau is displayed just before
  \emph{pivoting} on $a_{jj}^S$, where 
$j \in S^c$. That is, \hl{just before} % I changed ``for'' by "just before'' to better convey the meaning of the sentence
moving variable $x_j^0$ \emph{out} of the basis, and~
carrying $x_j^1$ \emph{into} the basis, \mbox{which corresponds} to changing
  from the $S$- to the $S \cup \{j\}$-active BFS.
After the pivot step, 
the updated tableau is shown in Table~\ref{tab:tabl2}.

One can readily use such tableaux to implement the 
top-down adaptive-greedy algorithm $\mathrm{AG}_{\mathcal{F}}^{\textup{TD}}$ in
Algorithm \textup{\ref{fig:agafref}}, by~first constructing the initial
tableau for the $\emptyset$-active policy, and~then
carrying out $n = |\mathcal{N}|$ pivot steps, at~each of which the
former BFS active set is augmented by a state. Such a
direct approach results in an implementation that we
call the 
\emph{Conventional-Pivoting Index} (CP) algorithm.
An immediate counting argument shows that the $n$ pivot steps of
that algorithm perform $2 n^3 + O(n^2)$ arithmetic operations---without considering
computation of the initial tableau, which will be addressed
\mbox{in Section~\ref{s:it} below}.

Note that this algorithm can also be applied to a restless bandit
 instance to test numerically whether it is
indexable. The~project will be indexable if, and only if, the successive pivot
steps, as~the price parameter $\lambda$ decreases from $\infty$ to $-\infty$ in the
parametric-objective simplex algorithm of~\cite{gassSaaty55}, \mbox{can be}
performed augmenting the current BFS by adding a state, thus
producing a nested active-set family $\mathcal{F}_0 = \{S_0, S_1,
\ldots, S_n\}$ with $S_0 = \emptyset \subset \cdots \subset S_n =
\mathcal{N}$. 

\begin{table}[H]
\caption{\hl{Simplex tableau (parametric) fo}r $S$-active%Please confirm if we can change it to the three-line table format? if not, please explain.
% I'm not really sure what is meant here by the three-line table format, but the table should be left like it is formatted now. The reason is that this is a parametric simplex tableau, so the lines shown are standard and should not  be changed
  BFS, with~pivot $a_{jj}^S$.}
\begin{center}
\begin{tabular}{cccc} 
& $\transp{(\mathbf{x}_S^0)}$ & $x_j^1$ & 
  $\transp{(\mathbf{x}_{S^c \setminus \{j\}}^1)}$ 
\\  \cmidrule{2-4} 
$\mathbf{x}_S^1$  & 
\multicolumn{1}{|c}{$\mathbf{A}_{SS}^S$} &
  $\mathbf{A}_{Sj}^S$
 & \multicolumn{1}{c|}{$\mathbf{A}_{S S^c \setminus \{j\}}^S$} \\
$x_j^0$  & 
\multicolumn{1}{|c}{$\mathbf{A}_{jS}^S$} & ${\boxed{a_{jj}^S}}$ &
\multicolumn{1}{c|}{$\mathbf{A}_{j S^c \setminus \{j\}}$} \\
$\mathbf{x}_{S^c \setminus \{j\}}^0$ & 
\multicolumn{1}{|c}{$\mathbf{A}_{S^c \setminus \{j\}, S}^S$} & 
$\mathbf{A}_{S^c \setminus \{j\}, j}^S$ & 
\multicolumn{1}{c|}{$\mathbf{A}_{S^c \setminus \{j\} S^c \setminus \{j\}}^S$}
\\ \cmidrule{2-4} 
& \multicolumn{1}{|c}{$\mathbf{g}_S^S$} & $-g_j^S$ & \multicolumn{1}{c|}{$- \mathbf{g}_{S^c \setminus \{j\}}^S$}
\\
& \multicolumn{1}{|c}{$\mathbf{f}_S^S$} & $-f_j^S$ &
\multicolumn{1}{c|}{$- \mathbf{f}_{S^c \setminus \{j\}}^S$}
\\ \cmidrule{2-4} 
\end{tabular}
\end{center}
\label{tab:invtabl}
\end{table}
\unskip

\begin{table}[H]
\caption{\hl{Tableau for} $S \cup \{j\}$-active  %Please confirm if we can change it to the three-line table format? if not, please explain.
% I'm not really sure what is meant here by the three-line table format, but the table should be left like it is formatted now. The reason is that this is a parametric simplex tableau, so the lines shown are standard and should not  be changed
  BFS, after~pivoting on $a_{jj}^S$.}
\begin{center}
\begin{tabular}{cccc} 
& $\transp{(\mathbf{x}_S^0)}$ & $x_j^0$ & 
  $\transp{(\mathbf{x}_{S^c \setminus \{j\}}^1)}$ 
\\  \cmidrule{2-4} 
$\mathbf{x}_S^1$ & 
\multicolumn{1}{|c}{$\displaystyle \mathbf{A}_{SS}^S - (a_{jj}^S)^{-1} \mathbf{A}_{Sj}^S
    \mathbf{A}_{jS}^S$} &
  $\displaystyle -(a_{jj}^S)^{-1} \mathbf{A}_{Sj}^S$
 & \multicolumn{1}{c|}{$\displaystyle \mathbf{A}_{S S^c \setminus \{j\}}^S - 
   (a_{jj}^S)^{-1} \mathbf{A}_{Sj}^S \mathbf{A}_{j S^c \setminus \{j\}}^S$}  \\
$x_j^1$ & 
\multicolumn{1}{|c}{$\displaystyle (a_{jj}^S)^{-1} \mathbf{A}_{jS}^S$} & $\displaystyle (a_{jj}^S)^{-1}$ & \multicolumn{1}{c|}{$\displaystyle (a_{jj}^S)^{-1} \mathbf{A}_{j S^c \setminus \{j\}}^S$} \\
$\mathbf{x}_{S^c \setminus \{j\}}^0$ & 
\multicolumn{1}{|c}{$\displaystyle \mathbf{A}_{S^c \setminus \{j\}, S}^S - 
 (a_{jj}^S)^{-1} \mathbf{A}_{S^c \setminus \{j\}, j}^S
   \mathbf{A}_{jS}^S$} & 
$\displaystyle - (a_{jj}^S)^{-1} \mathbf{A}_{S^c \setminus \{j\}, j}^S$ & 
\multicolumn{1}{c|}{$\displaystyle \mathbf{A}_{S^c \setminus \{j\} S^c \setminus \{j\}}^S - 
(a_{jj}^S)^{-1} \mathbf{A}_{S^c \setminus \{j\}, j}^S \mathbf{A}_{j S^c \setminus \{j\}}^S$} 
\\ \cmidrule{2-4} 
& \multicolumn{1}{|c}{$\displaystyle \mathbf{g}_S^S + 
 g_j^S (a_{jj}^S)^{-1} \mathbf{A}_{jS}^S$} & $\displaystyle g_j^S (a_{jj}^S)^{-1}$ & \multicolumn{1}{c|}{$\displaystyle - \mathbf{g}_{S^c \setminus \{j\}}^S + 
 g_j^S (a_{jj}^S)^{-1} \mathbf{A}_{j S^c \setminus \{j\}}^S$}
\\
& \multicolumn{1}{|c}{$\displaystyle \mathbf{f}_S^S + 
 f_j^S (a_{jj}^S)^{-1} \mathbf{A}_{jS}^S$} & $\displaystyle f_j^S (a_{jj}^S)^{-1}$ & \multicolumn{1}{c|}{$\displaystyle - \mathbf{f}_{S^c \setminus \{j\}}^S + 
 f_j^S (a_{jj}^S)^{-1} \mathbf{A}_{j S^c \setminus \{j\}}^S$}
\\ \cmidrule{2-4} 
\end{tabular}
\end{center}
\label{tab:tabl2}
\end{table}

We now consider how  to improve the efficiency of the CP algorithm, drawing on the observation that 
the tableau's rows
for basic variables $\mathbf{x}_S^1$ are not used
to update the reduced costs. 
Thus, it is enough to update and store the \emph{reduced tableaux}
shown in Table~\ref{tab:redtabl}.  
Table~\ref{tab:tabl2} demonstrates that a reduced tableau
can be updated without using the rows that have been deleted. 
The resulting simplification of the CP algorithm yields the 
\emph{Reduced-Pivoting} (RP) algorithm.
By an elementary counting argument, it is readily seen that the RP algorithm
carries out the $n$ pivot steps in $n^3 + O(n^2)$ 
operations, \mbox{thus
improving} by a factor of 2 the arithmetic operation complexity of
the CP~algorithm.

\begin{table}[H]
\caption{\hl{Reduced tableau} for $S$-active  %Please confirm if we can change it to the three-line table format? if not, please explain.
% I'm not really sure what is meant here by the three-line table format, but the table should be left like it is formatted now. The reason is that this is a parametric simplex tableau, so the lines shown are standard and should not  be changed
  BFS, with~pivot $a_{jj}^S$.}
\begin{center}
\begin{tabular}{cccc} 
& $\transp{(\mathbf{x}_S^0)}$ & $x_j^1$ & $\transp{(\mathbf{x}_{S^c \setminus \{j\}}^1)}$
\\  \cmidrule{2-4} 
$x_j^0$ & \multicolumn{1}{|c}{$\mathbf{A}_{jS}^S$} & ${\boxed{a_{jj}^S}}$ &
\multicolumn{1}{c|}{$\mathbf{A}_{j S^c \setminus \{j\}}^S$} \\
$\mathbf{x}_{S^c \setminus \{j\}}^0$ & \multicolumn{1}{|c}{$\mathbf{A}_{S^c \setminus
  \{j\}, S}^S$} & $\mathbf{A}_{S^c \setminus \{j\}, j}^S$ & \multicolumn{1}{c|}{$\mathbf{A}_{S^c \setminus \{j\} S^c \setminus \{j\}}^S$}
\\ \cmidrule{2-4} 
& \multicolumn{1}{|c}{$\mathbf{g}_S^S$} & $-g_j^S$ &
\multicolumn{1}{c|}{$- \mathbf{g}_{S^c \setminus \{j\}}^S$} \\
& \multicolumn{1}{|c}{$\mathbf{f}_S^S$} & $-f_j^S$ &
\multicolumn{1}{c|}{$- \mathbf{f}_{S^c \setminus \{j\}}^S$} \\ \cmidrule{2-4} 
\end{tabular}
\end{center}
\label{tab:redtabl}
\end{table}

We can exploit the 
assumption of PCL$(\mathcal{F})$-indexability to reduce even further the operation count, 
by updating and storing only the \emph{minimal tableaux} in
Table~\ref{tab:redtabl2}.
Such a tableau for 
the $S \cup \{j\}$-active BFS is computed from that for the
$S$-active BFS in Table~\ref{tab:redtabl2},
 as displayed in Table~\ref{tab:rtscj1}.
This results in the \emph{fast-pivoting $($FP$)$ adaptive-greedy} index
  algorithm $\mathrm{FP}_{\mathcal{F}}$  in
Algorithm~\ref{tab:fast}.

\begin{table}[H]
\caption{\hl{Minimal tableau} for $S$-active~BFS.}
\begin{center}
\begin{tabular}{cc} 
&  $\transp{(\mathbf{x}_{S^c}^1)}$ \\ \cline{2-2}
$\mathbf{x}_{S^c}^0$  & \multicolumn{1}{|c|}{$\mathbf{A}_{S^c S^c}^S$}
\\ \cmidrule{2-2}
& \multicolumn{1}{|c|}{$\mathbf{g}_{S^c}^S$} \\ 
& \multicolumn{1}{|c|}{$\mathbf{f}_{S^c}^S$}  \\ \cmidrule{2-2}
\end{tabular}
\end{center}
\label{tab:redtabl2}
\end{table}
\unskip

\begin{table}[H]
\caption{\hl{Minimal tableau} for $S \cup \{j\}$-active %Please confirm if we can change it to the three-line table format? if not, please explain.
% I'm not really sure what is meant here by the three-line table format, but the table should be left like it is formatted now. The reason is that this is a parametric simplex tableau, so the lines shown are standard and should not  be changed
  BFS, obtained after pivoting on $a_{jj}^S$.}
\begin{center}
\begin{tabular}{cc} 
& $\transp{(\mathbf{x}_{S^c \setminus \{j\}}^1)}$
\\  \cmidrule{2-2}
$\mathbf{x}_{S^c \setminus \{j\}}^0$   & 
\multicolumn{1}{|c|}{$\displaystyle \mathbf{A}_{S^c \setminus \{j\} S^c \setminus \{j\}}^S - 
 (a_{jj}^S)^{-1} \mathbf{A}_{S^c \setminus \{j\}, j}^S
   \mathbf{A}_{j S^c \setminus \{j\}}^S$} 
   \\
\cmidrule{2-2}
& \multicolumn{1}{|c|}{$\displaystyle \mathbf{g}_{S^c \setminus \{j\}}^S -
 g_j^S (a_{jj}^S)^{-1} \mathbf{A}_{S^c \setminus \{j\}, j}^S$} \\
& \multicolumn{1}{|c|}{$\displaystyle \mathbf{f}_{S^c \setminus \{j\}}^S -
 f_j^S (a_{jj}^S)^{-1} \mathbf{A}_{S^c \setminus \{j\}, j}^S$} \\  \cline{2-2}
\end{tabular}
\end{center}
\label{tab:rtscj1}
\end{table}
\unskip

\begin{algorithm}[H]
\caption{\hl{The fast-pivoting adaptive-greedy index algorithm} $\mathrm{FP}_{\mathcal{F}}$.} %Please confirm whether this is an algorithm, if it is, please change it to the format of the algorithm
% Answer: Yes, this is an algorithm, I have changed the environment and the references to it accordingly
\begin{center}
\fbox{%
\begin{minipage}{\textwidth}
\textbf{Output:}
$\{j_k, \lambda^*_{j_k}\}_{k=1}^{n}$
\begin{tabbing}
\textbf{solve }
$
\mathbf{A}^{(0)}
\begin{pmatrix}
\mathbf{I}_{\mathcal{N}, \mathcal{N} \setminus \{j^*\}}
 -  \boldsymbol{\Psi}_{\mathcal{N}, \mathcal{N} \setminus \{j^*\}}^0 & 
\widetilde{\mathbf{m}}^{0}
\end{pmatrix}
= \begin{pmatrix}
\mathbf{I}_{ \mathcal{N},  \mathcal{N}\setminus \{j^*\}} -  
\boldsymbol{\Psi}_{ \mathcal{N},  \mathcal{N}\setminus \{j^*\}}^1 & 
\widetilde{\mathbf{m}}^{1}
\end{pmatrix}
$
\\
$\begin{pmatrix} \mathbf{g}^{(0)} \\
  \mathbf{f}^{(0)} \end{pmatrix} := 
\begin{pmatrix}
  \mathbf{q}^1 \\ \mathbf{r}^1 \end{pmatrix} -
  \begin{pmatrix}
 \mathbf{q}^0 \\ \mathbf{r}^0 \end{pmatrix} \mathbf{A}^{(0)}$; \, 
 $S_0 := \emptyset$
\end{tabbing}

\begin{tabbing}
\textbf{for} \= $k := 1$   \textbf{ to } $n$  {\bf do} \\
 \> 
 $\lambda_j^{(k-1)} := f_j^{(k-1)}/g_j^{(k-1)}, \quad j
\in \partial^{\textup{out}}_{\mathcal{F}} S_{k-1}$ \\
 \> \textbf{pick} 
 $\displaystyle j_{k} \in \argmax
      \{\lambda^{({k-1})}_i\colon
                j \in \partial^{\textup{out}}_{\mathcal{F}}
 S_{k-1}\}$; \, $\displaystyle \lambda_{j_k}^* := 
 \lambda^{({k-1})}_{j_k}$; \, $\displaystyle S_{k} := S_{k-1} \cup \{j_{k}\}$  \\
 \> \textbf{if} \= $k < n$ \textbf{then} \\
\>  \> $\mathbf{A}_{S_{k}^c j_k}^{({k})} := (1/ a_{j_k j_k}^{({k-1})})
 \mathbf{A}_{S_{k}^c j_k}^{({k-1})}$; \,
$\displaystyle \mathbf{A}_{S_k^c S_{k}^c}^{(k)}
 := \mathbf{A}_{S_k^c S_{k}^c}^{({k-1})} - \mathbf{A}_{S_{k}^c 
 j_k}^{({k})} \mathbf{A}_{j_k S_{k}^c}^{({k-1})}$
\\
 \>  \textbf{end } \{ if \} \\
 \>  $\displaystyle \mathbf{g}_{S_k^c}^{(k)} := \mathbf{g}_{S_k^c}^{({k-1})} - 
  g_{j_k}^{({k-1})} \mathbf{A}_{S_k^c j_k}^{({k})}$; \, $\displaystyle \mathbf{f}_{S_k^c}^{(k)} := \mathbf{f}_{S_k^c}^{({k-1})} - 
  f_{j_k}^{({k-1})} \mathbf{A}_{S_k^c j_k}^{({k})}$ \\
\textbf{end} \{ for \}
\end{tabbing}
\end{minipage}}
\end{center}
\label{tab:fast}
\end{algorithm}

The next result evaluates the  operation count of
algorithm $\mathrm{FP}_{\mathcal{F}}$, showing that it outperforms significantly
that of algorithm RP. Note that the complexity of its
loop---which performs the $n$ pivot steps---matches that of Gaussian elimination for solving an $n \times n$ system of linear equations, and~is hence unlikely that such complexity can be improved for general restless~bandits.

\begin{proposition}
\label{pro:aga}
The loop of algorithm $\mathrm{FP}_{\mathcal{F}}$ entails $(2/3)
n^3 + O(n^2)$ arithmetic operations.
\end{proposition}
\begin{proof}
The operation count in the loop is dominated by the update of matrix $\mathbf{A}_{S_k^c
  S_k^c}^{S_k}$ at each step $k$, taking $2 (n-k)^2$ arithmetic
operations.  This yields the total  operation count as stated.
\end{proof}

\subsection{Computing the Initial~Tableau}
\label{s:it}
We next address computation of the 
initial tableau, which corresponds to the 
$\emptyset$-active BFS, in~a form that is numerically stable and that applies both to the discounted
  and the average criterion to be addressed in
Section~\ref{s:rac}. The~ tableaux  for the average criterion arise as limits letting the
discount rate $\alpha$ vanish in the discounted~tableaux. 

Note that (cf.\ (\ref{eq:matdef}))
\begin{equation}
\label{eq:mats0}
\mathbf{B}^\emptyset = \transp{(\mathbf{I} -
  \boldsymbol{\Psi}^0)}, 
\quad
\mathbf{N}^\emptyset = 
\transp{(\mathbf{I} -  \boldsymbol{\Psi}^1)}, \quad
\mathbf{H}^\emptyset = (\mathbf{B}^\emptyset)^{-1}, \quad
\mathbf{A}^\emptyset =  \mathbf{H}^\emptyset \mathbf{N}^\emptyset.
\end{equation}
Thus, a~straightforward approach to computing $\mathbf{A}^\emptyset$ is to solve the linear equation system
\begin{equation}
\label{eq:aemples}
\transp{(\mathbf{A}^\emptyset)} (\mathbf{I} -  \boldsymbol{\Psi}^0)
  = (\mathbf{I} -  \boldsymbol{\Psi}^1).
\end{equation}
However, this has a disadvantage: as $\alpha$ vanishes, the~ matrices $\mathbf{I} - \boldsymbol{\Psi}^a$ become increasingly
 ill-conditioned, as~they are singular for $\alpha = 0$---because they converge to
 $\mathbf{I} - \mathbf{P}^a$ where $\mathbf{P}^a \triangleq
 (p_{ij}^a)$.

To overcome such a drawback, we draw on the identity 
$(\mathbf{I} - \boldsymbol{\Psi}^a) \mathbf{1}
= \mathbf{1} - \boldsymbol{\Psi}^a$, 
which is a consequence from (\ref{eq:phiialst}).
From this and (\ref{eq:mats0}), we have 
\[
\transp{(\mathbf{A}^\emptyset)} (\mathbf{1} - \boldsymbol{\Psi}^0) = 
\mathbf{1} - \boldsymbol{\Psi}^1.
\]
The latter identity  gives a useful
counterpart as $\alpha \searrow 0$.
 Thus, writing as $\xi_i^{a}$ the length of an $(i, a)$-stage (cf.\
 Section~\ref{s:dtr}), and~applying the MacLaurin series 
\[
\psi_{i}^{a} = \Ex\big[e^{-\alpha \xi_i^{a}}\big] = 1 -
\alpha m_i^{a} + O(\alpha^2) \quad \text{as }
\alpha \searrow 0,
\]
one obtains in the limit 
\[
\transp{(\mathbf{A}^\emptyset)} \mathbf{m}^{0} = 
\mathbf{m}^{1},
\]
with $m_i^a$  the mean
length of an $(i, a)$-stage and $\mathbf{m} = (m_i^a)_{i \in \mathcal{N}}$.

We thus obtain the following approach to
computing the initial tableau, for~$\alpha \geqslant 0$.
Letting
\[
\widetilde{m}_i^{a} \triangleq
\begin{cases}
m_i^{a} & \text{ if } \alpha = 0 \\
(1-\psi_{i}^{a})/\alpha & \text{ if } \alpha > 0,
\end{cases}
\]
and $\widetilde{\mathbf{m}}^a = (\widetilde{m}_i^{a})_{i \in \mathcal{N}}$,
pick any state $j^* \in \mathcal{N}$ and solve the 
\emph{(block) linear equation system}
\begin{equation}
\label{eq:neweqs}
\transp{\begin{pmatrix}
\mathbf{I}_{ \mathcal{N},  \mathcal{N}\setminus \{j^*\}}
 -  \boldsymbol{\Psi}_{ \mathcal{N},  \mathcal{N}\setminus \{j^*\}}^0 & 
\widetilde{\mathbf{m}}^{0}
\end{pmatrix}} \mathbf{A}^\emptyset
= 
\transp{\begin{pmatrix}
\mathbf{I}_{ \mathcal{N},  \mathcal{N}\setminus \{j^*\}} -  
\boldsymbol{\Psi}_{ \mathcal{N},  \mathcal{N}\setminus \{j^*\}}^1 & 
\widetilde{\mathbf{m}}^{1}
\end{pmatrix}}
\end{equation}
to obtain $\mathbf{A}^\emptyset$.
Then, calculate initial reduced costs using (\ref{eq:mats0}) and Lemma \ref{lma:pmchar} (d, e) as
\begin{equation}
\label{eq:initrc}
\begin{split}
\mathbf{g}^\emptyset & = \mathbf{q}^1 - 
  \mathbf{q}^0  \mathbf{A}^\emptyset \\
\mathbf{f}^\emptyset & = \mathbf{r}^1 - \mathbf{r}^0 \mathbf{A}^\emptyset.
\end{split}
\end{equation}

\section{Extension to the Average~Criterion}
\label{s:rac}
In applications where the (long-run)
average criterion is employed, one must consider the version of
$\lambda$-price problem (\ref{eq:nuwpro})  based on the average reward and resource metrics given by
\begin{equation}
\label{eq:avfipi}
F_i^\pi 
\triangleq \liminf_{T \nearrow \infty} \frac{1}{T}
\Ex_i^\pi\left[\int_{0}^T R_{X(t)}^{A(t)} \, dt\right] = \liminf_{K
  \nearrow \infty} \frac{1}{K} \Ex_i^\pi\left[\sum_{k=0}^K r_{X_k}^{a_k} \right],
\end{equation}
and
\begin{equation}
\label{eq:avgipi}
G_i^\pi 
\triangleq \limsup_{T \nearrow \infty} \frac{1}{T} \Ex_i^\pi\left[\int_{0}^T Q_{X(t)}^{A(t)} 
   \, dt\right] = \limsup_{K
  \nearrow \infty} \frac{1}{K}  \Ex_i^\pi\left[\sum_{k=0}^K
  q_{X_k}^{a_k}\right].
\end{equation}

As in (\cite{nmmp02}, Section~6.5), we assume that the
embedded process $X_k$ is \emph{communicating}, so each state can
be reached from every other state under some stationary policy. 
Under this assumption, the above metrics do not depend on
the initial state under stationary deterministic policies, so one can
write $F^S$ and $G^S$ for active sets $S \subseteq \mathcal{N}$. This allows a
straightforward extension of the above indexability theory  to the average~criterion.

Regarding the  algorithms discussed above, they apply without change to the
average criterion, as~the results presented in Section~\ref{s:it} show that the
corresponding tableaux are simply the limits of the
discounted tableaux as the discount rate goes to zero, and~
also outline how to evaluate the initial tableau.
To properly extend the results, one must
further assume that the active-set family $\mathcal{F}$ satisfies that, for~every $S \in \mathcal{F}$, the~$S$-active
policy is \emph{unichain}, so it induces  a single recurrent class on the embedded process
$X_k$ plus a class of
transient states, which may be~empty.

\section{Numerical~Experiments}
\label{s:ce}
This section discusses results of numerical experiments,
based on implementations by the author of the aforementioned~algorithms.

\subsection*{Comparing Runtimes of Index~Algorithms}
\label{s:rcia}
The runtime performance of an algorithm depends not only on its arithmetic operation count, but~also on its memory-access patterns, which can actually be the dominant
factor.
Hence, to~compare the  algorithms considered herein,  a~numerical study has been conducted, using MATLAB implementations
  developed by the author. 
The
experiments were run on a PC with an Intel Core \mbox{i7-8700 CPU} at 3.2 GHz with 16 GB of RAM
using MATLAB 2020b under Windows 10 Enterprise.
\mbox{For the} state space sizes $n = 1000, 2000, \ldots,$ 15,000,
 a   discrete-time restless bandit instance 
was randomly constructed.
Transition matrices were generated from random matrices with 
Uniform$[0, 1]$ entries, dividing each row by its sum. 
Active rewards were randomly generated with Uniform$[0, 1]$ entries,
and passive rewards were zero.
The discount factor was $\beta = 0.8$.

For each generated instance, 
the CP algorithm was first used to test for
indexability and for PCL-indexability (by checking the signs of
marginal resource metrics for the  nested active-set
family obtained).
Since such tests were positive in each instance,
the Whittle index was calculated using the CP, RP, and
FP algorithms (taking $\mathcal{F} = 2 ^\mathcal{N}$ in the latter). 

Table~\ref{tab:numex1} records the runtimes for the loop of each algorithm, without~counting the initialization stage of computing the initial tableau, while 
Figure~\ref{fig:rclsf} plots them, along with cubic least-squares fitting curves.
These results highlight that the FP algorithm, whose loop operation count is of
$(2/3) n^3 + O(n^2)$, is indeed the fastest algorithm, followed by the CP and RP algorithms.
Recall that $2 n^3+ O(n^2)$ and $n^3+ O(n^2)$ are the loop operation counts for the CP and the
RP~algorithms.

 \begin{table}[H]
\caption{Runtimes (sec.) of index~algorithms.} \label{tab:numex1}
\centering
\begin{tabular}{rrrr}
\toprule
\multicolumn{1}{c}{$n$} & \multicolumn{1}{c}{FP}  &  \multicolumn{1}{c}{RP} &  \multicolumn{1}{c}{CP} \\ \midrule
$1000$   &        $1.8$ &       $2.6$ &        $2.2$ \\
$2000$  &       $14.8$ &       $23.6$ &      $19.6$ \\
$3000$  &   $53.3$ &      $77.1$  &    $67.8$ \\
$4000$   &   $121.4$  & $193.4$ &    $156.3$\\
$5000$  &   $227.6$ &   $342.8$ &  $295.9$ \\
$6000$  &   $433.2$  & $647.9$ &    $541.8$ \\
$7000$ &     $699.6$  & $1034.8$ &  $862.3$ \\
$8000$ &   $1118.3$   &  $1531.5$ &  $1280.9$ \\
$9000$ &    $1530.8$  & $2173.6$  &  $1822.3$ \\
10,000 & $2100.7$   &  $2919.3$  &  $2500.9$ \\
11,000 &  $2687.2$  & $3580.0$   &  $3277.7$ \\
12,000   &  $3575.9$  & $5055.7$   &  $4300.6$ \\
13,000 &  $4747.4$ &  $6610.4$ &  $5539.2$ \\
14,000 &  $5629.5$ &   $7923.1$ &  $6829.9$ \\
15,000 &  $7254.1$ &   $8871.0$ &  $8250.5$
\\ \bottomrule
\end{tabular}
\end{table}
\unskip

 \vspace{-12pt}
\begin{figure}[H]
\centering
\includegraphics[height=3in,width=\textwidth,keepaspectratio]{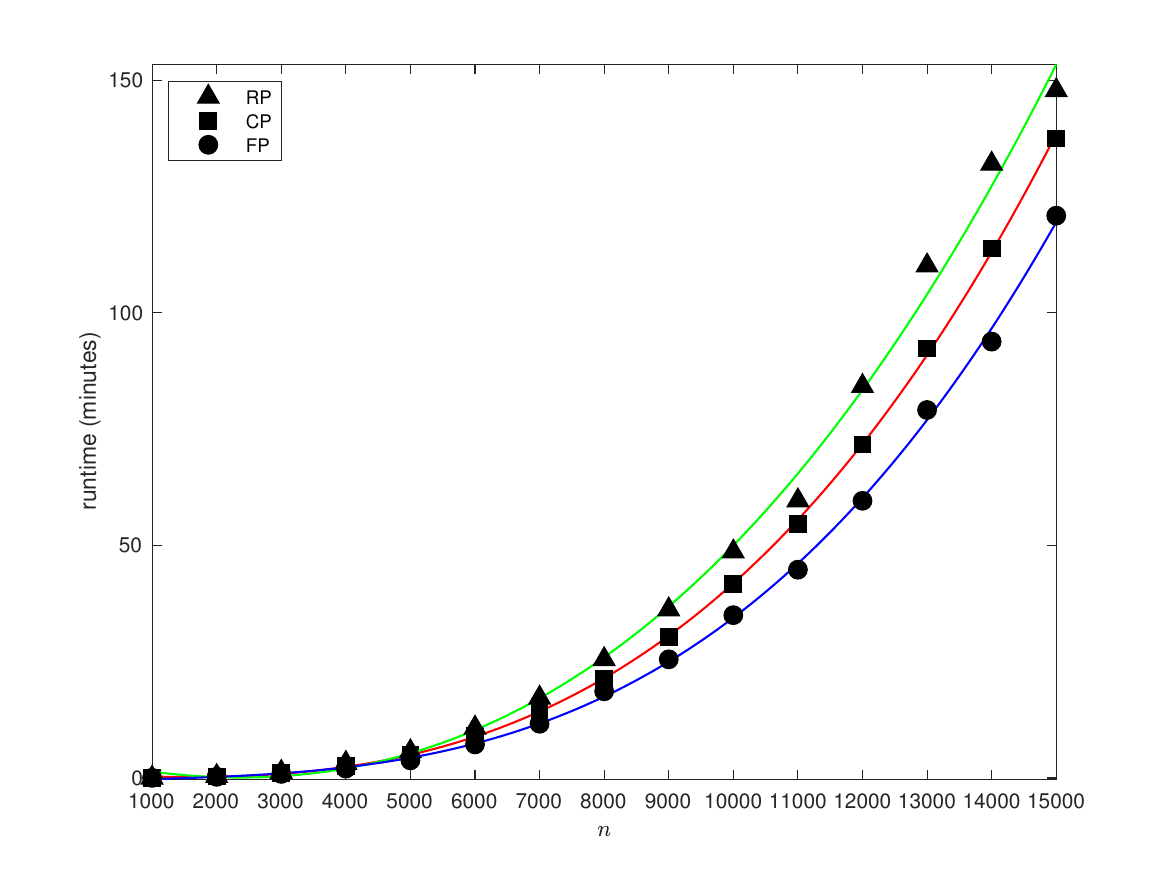}%
\vspace{-12pt}

\caption{Runtimes of index algorithms vs.\ number of project states with cubic least-squares~fits.}
\label{fig:rclsf}
\end{figure}

The observed discrepancies between operation count complexity and actual runtime performance are \hl{explained} by taking into account
% I changfd for clarity and to avoid redundancy ``accounted for'' by ``explained''
 the memory-access patterns of the algorithms.
Algorithm CP, which carries out conventional pivot steps, has
 efficient memory-access patterns, because the
coefficient matrix $\mathbf{A}$ is always updated as a memory block of contiguous storage.
Yet, both algorithms RP and
FP  achieve a reduction in operations by
operating on submatrices of $\mathbf{A}$,  with~resulting noncontiguous memory-access
patterns that are time-consuming. 
However, for~the FP algorithm, as~the decrease in arithmetic
operations is substantial, it compensates such inefficiencies,  and~comes out in practice as the 
fastest~algorithm. 

\section{Discussion}
\label{s:disc}
This paper has presented a new algorithm for computing the Whittle index of a general finite-state semi-Markov restless bandit, based on an efficient implementation of the adaptive-greedy algorithmic scheme introduced in~\cite{nmaap01,nmmp02} for restless bandits, in~which it was not specified how to evaluate certain metrics arising in the algorithm description. 
The algorithm extends to restless bandits the fast-pivoting implementation developed in~\cite{nmijoc07} for classic (non-restless) bandits, and~results from a similar approach, exploiting the structure of parametric simplex tableaux to reduce the operation count down to $(2/3) n^3 + O(n^2)$---apart from the computational effort to compute the initial tableau. 
It is unlikely that such a complexity count can be improved for general restless bandits, as~$(2/3) n^3 + O(n^2)$ is the complexity of Gaussian elimination for solving an $n \times n$ linear equation system, and~it is shown in~\cite{nmaap01,nmmp02} that computing the Whittle index entails at least solving an equation system with such dimensions (\mbox{albeit with} an ordering of the states that generally is not known in advance).

The complexity of the proposed fast-pivoting algorithm is the best that has been reported in the literature. In~contrast, the~Whittle index algorithm recently presented in~\cite{ayestetal20}---for  the average criterion---has a complexity of $O(n^4 2^n)$, which reduces to $O(n^5)$ for \hl{one-dimensional} 
% I added this for clarity
indexable bandits provided it is known that 
 threshold policies are~optimal \hl{for them}. % I added this for clarity
 
 The new algorithm presented herein, whose implementation is straightforward, will be most useful for computing the Whittle index in complex models with large-scale, multi-dimensional state spaces for which closed index formulae cannot be derived, and~an efficient computational approach is needed.
 Given the explosion of research interest in restless bandit models in the last decade, the~proposed algorithm thus has the potential of becoming a useful tool, allowing researchers to expand the scope of Whittle's index policy to large-scale complex~models.

Future research directions include developing efficient implementations with substantially lower complexity by exploiting the special structure of relevant model classes arising in applications, and~testing the algorithm in large-scale real world models with real data.
Another avenue of research is to develop software implementations that improve the efficiency of the computationally costly block matrix operations required by the fast-pivoting~algorithm.

\section{Conclusions}
\label{s:concl}
To conclude, the~findings of this paper can be summarized as~follows:
\begin{itemize}[leftmargin=2.2em,labelsep=5.5mm]
\item[-] A new algorithm to compute the Whittle index of a general $n$-state semi-Markov restless bandit is presented, which can also be used to test numerically for indexability. After~an initialization step, the~algorithm computes the $n$ index values in an $n$-step loop with a complexity of $(2/3) n^3 + O(n^2)$ arithmetic operations.
\item[-] The algorithm extends to Whittle's index the fast-pivoting $(2/3) n^3 + O(n^2)$ algorithm introduced by the author in~\cite{nmijoc08}  \hl{for the Gittins  index of classic (non-restless) banditss}, which also has the lowest \mbox{complexity to date}. % I changed for clarity ``for classic (non-restless) bandits'' to ``for the Gittins  index of classic (non-restless) bandits''
\item[-] The proposed algorithm has substantially better complexity than alternative algorithms proposed in the literature.
\item[-] The algorithm will be especially useful for computing the Whittle index in large-scale multi-dimensional models where the index cannot be derived in closed form and alternative algorithms will result in prohibitive computation times.
\end{itemize}

%%%%%%%%%%%%%%%%%%%%%%%%%%%%%%%%%%%%%%%%%%

%%%%%%%%%%%%%%%%%%%%%%%%%%%%%%%%%%%%%%%%%%
% \authorcontributions{\hl{All authors have read and agreed to the published version of the manuscript}.}%please add detailed Author Contributions for every author.

\section*{Funding}
This research has been developed over a number of years, and~has been funded by the Spanish Government under grants MEC MTM2004-02334 and  PID2019-109196GB-I00/AEI/10.13039/501100011033.

%%%%%%%%%%%%%%%%%%%%%%%%%%%%%%%%%%%%%%%%%%
\section*{Acknowledgments}
Preliminary early versions of this work were published by the author in the abridged proceedings~\cite{jnmsmct07b} of the Second International Workshop on Tools for Solving Structured
Markov Chains (\mbox{SMCtools 2007}), Nantes,
France, 2007, and~in the working paper~\cite{jnmsmct07a}.

%=====================================
% References, variant B: internal bibliography
%=====================================

\end{document}